\documentclass[10pt,reqno]{amsart}

\usepackage{amssymb,amsmath,amsthm,amsxtra,amsfonts}
\usepackage{graphicx}

\usepackage{hyperref}

\usepackage{setspace}

\newtheorem{theorem}{Theorem}[section]

\newtheorem{corollary}[theorem]{Corollary}
\newtheorem{lemma}[theorem]{Lemma}

\numberwithin{equation}{section}
\newtheorem{definition}[theorem]{Definition}

\theoremstyle{remark}
\newtheorem*{remarks}{Remarks}
\newtheorem*{remark}{Remark}

\newcommand{\bbC}{{\mathbb{C}}}
\newcommand{\bbD}{{\mathbb{D}}}

\newcommand{\bbR}{{\mathbb{R}}}

\newcommand{\bbZ}{{\mathbb{Z}}}

\newcommand{\fre}{{\frak{e}}}
\newcommand{\frakR}{{\mathfrak{R}}}

\newcommand{\calE}{{\mathcal{E}_R}}

\newcommand{\calJ}{{\mathcal J}}

\newcommand{\calS}{{\mathcal S}}
\newcommand{\calT}{{\mathcal T}}

\newcommand{\bdone}{{\boldsymbol{1}}}
\newcommand{\bdnot}{{\boldsymbol{0}}}

\newcommand{\til}{\tilde  }

\newcommand{\sgn}{\text{\rm{sgn}}}

\newcommand{\supp}{\text{\rm{supp}}}

\newcommand{\beq}{\begin{equation}}
\newcommand{\eeq}{\end{equation}}
\newcommand{\ba}{\begin{align*}}
\newcommand{\ea}{\end{align*}}
\newcommand{\veps}{\varepsilon}

\newcommand{\RR}{\frakR_{\eta,C,\gamma}}

\DeclareMathOperator{\real}{Re}
\DeclareMathOperator{\imag}{Im}

\DeclareMathOperator*{\esssup}{ess\,supp}

\DeclareMathOperator*{\res}{Res}
\DeclareMathOperator{\Int}{Int}
\DeclareMathOperator{\sg}{sg}

\begin{document}

\title[Spectral and resonance problem for Jacobi operators]{Spectral and Resonance Problem for Perturbations of Periodic Jacobi Operators}
\author{Rostyslav Kozhan}
\thanks{
The project was partially supported by the grant KAW 2010.0063 from the Knut and Alice Wallenberg Foundation}
\email{kozhan@kth.se}
\address{Royal Institute of Technology; Department of Mathematics; Stockholm, Sweden}


\begin{abstract}
Necessary and sufficient conditions are presented for a measure to be the spectral measure of a finite range or exponentially decaying perturbation of a periodic Jacobi operator.

As a corollary we can fully solve the inverse resonance problem: given resonances and eigenvalues we can recover the spectral measure of the Jacobi operator; we provide necessary and sufficient conditions under which such an operator exists and is unique; and we show that the inverse resonance problem is stable under small perturbations.
\end{abstract}

\maketitle






\section{Introduction}\label{sIntro}

\subsection{Preliminaries}\label{ssIntroPrelim}
By a periodic Jacobi operator/matrix we will call an operator on $\ell_2(\bbZ_+)$ of the form
\begin{equation}\label{jacobi}
\calJ=\left(
\begin{array}{cccc}
b_1&a_1&{0}&\\
a_1 &b_2&a_2&\ddots\\
{0}&a_2 &b_3&\ddots\\
&\ddots&\ddots&\ddots\end{array}\right)
\end{equation}
for which there exists some (period) $p\ge1$ such that
\begin{equation}\label{periodic}
a_{n+p}=a_n, \quad b_{n+p}=b_n \quad \mbox{for all } n.
\end{equation}

Any operator of the form \eqref{jacobi} will also be denoted by $(a_n,b_n)_{n=1}^\infty$. Sequences $\{a_n\}$, $\{b_n\}$ are called the Jacobi parameters of $\calJ$. We always assume $a_n>0$ and $b_n\in\bbR$ for all $n$.

We will be concerned here with two classes of perturbations of periodic Jacobi matrices which we will denote by~\eqref{ev} and~\eqref{exp}. The first consists of the eventually periodic matrices, that is, all matrices $\calJ$ for which there exists $s\ge1$ such that
\begin{equation}\tag{$P_{f.r.}$}\label{ev}
a_n=a_n^{\circ}, \quad b_n=b_n^{\circ} \quad \mbox{for all }n\ge s,
\end{equation}
for some periodic Jacobi matrix $(a_n^{\circ},b_n^{\circ})_{n=1}^\infty$. Another term we will use for these is finite range perturbations.

The second class consists of the exponentially decaying perturbations of periodic matrices, i.e., those matrices for which there exists $1 <  R \le\infty$ such that
\begin{equation}\tag{$P_R$}\label{exp}
\limsup_{n\to\infty} \left(|a_n-a_n^{\circ}|+|b_n-b_n^{\circ}|\right)^{1/2n}\le R^{-1},
\end{equation}
for some periodic Jacobi matrix $(a_n^{\circ},b_n^{\circ})_{n=1}^\infty$. 

Let $\mu$ be the spectral measure of $\calJ$ with respect to the vector $\delta_1=(1,0,0,\ldots)^T$ (which is cyclic since all $a_j>0$):
\begin{equation}\label{spectralDefinition}
\int_\bbR f(x)d\mu(x) = \langle \delta_1, f(\calJ) \delta_1 \rangle.
\end{equation}

It is well-known that the essential spectrum of a $p$-periodic Jacobi matrix consists of $p$ closed intervals which are allowed to touch but otherwise are disjoint. Thus
\begin{equation}\label{spectrum}
\sigma_{ess}(\calJ)\equiv \esssup\mu = \bigcup_{j=1}^{g+1} [\alpha_j,\beta_j] \equiv \fre, \quad \alpha_j<\beta_j <\alpha_{j+1}.
\end{equation}
We will refer to $[\alpha_j,\beta_j]$ as a ``band'', and to $(\beta_j,\alpha_{j+1})$ as a ``gap''. $g$ here is the number of gaps (in general $g \le p-1$). We call $\fre$ a finite gap set.

By Weyl's theorem on compact perturbations, the essential spectrum of Jacobi matrices from~\eqref{ev} or~\eqref{exp} coincides with \eqref{spectrum}.

An important special case of periodic Jacobi matrices is with $p=1$ ($g=0$). Typically one normalizes, $a_n=1$, $b_n=0$ for all $n$, in which case the spectrum is one interval $[-2,2]$ and the spectral measure is $\tfrac{\sqrt{4-x^2}}{2\pi} dx$. We will refer to this matrix as the free Jacobi matrix. 
We will use~\eqref{evFree} 
to refer to the class of finite range perturbations of the free Jacobi matrix (or ``eventually free'')
\begin{equation}\tag{$P^{p=1}_{f.r.}$}\label{evFree}
a_n=1, \quad b_n=0 \quad \mbox{for all }n\ge s,
\end{equation}
and~\eqref{expFree} 
to refer to the class of exponentially decaying perturbations of the free Jacobi matrix
\begin{equation}\tag{$P^{p=1}_R$}\label{expFree}
\limsup_{n\to\infty} \left(|a_n-1|+|b_n|\right)^{1/2n}\le R^{-1}.
\end{equation}

Volumes of literature have been devoted to investigating  various classes of perturbations of the free Jacobi operator. Of a special interest is the relation between the properties of Jacobi coefficients and of the spectral measure, especially in the cases when the relation is if-and-only-if. These cases are rare (for a great overview, see Simon's~\cite[Chapt 1]{Rice}). We would like to distinguish the Killip--Simon theorem for $L^2$ perturbations~\cite{KS}, Geronimo--Nevai's~\cite{GN} and Ryckman's~\cite{RyckBax} papers for weighted $L^1$ perturbations, and Ryckman's~\cite{RyckSz} theorem for $H^{1/2}$ perturbations.  Of these only the Killip--Simon theorem was generalized to the periodic setting by Damanik--Killip--Simon~\cite{DKS}. The current paper gives the if-and-only-if characterization for exponentially decaying and finite range perturbations of the free and periodic Jacobi matrices.

We would like to remark that the theory of orthogonal polynomials on the \textit{unit circle} (OPUC) usually goes in parallel with (and often ahead of) the spectral theory of Jacobi operators. Indeed, the above-mentioned results of Killip--Simon, Geronimo--Nevai, and Ryckman have famous predating analogues in OPUC: Szeg\H{o}'s theorem~\cite{Sze15} (see also Verblunsky~\cite{Ver36}), Baxter's theorem~\cite{Bax61}, and strong Szeg\H{o}'s theorem~\cite{Sze52} (see also Ibragimov~\cite{Ibr68}), respectively. The OPUC analogue of the current results is not yet known (though the Nevai--Totik~\cite{NT} and Peherstorfer--Steinbauer~\cite{PehSte86b} results are not unrelated), and will be addressed by the author in an upcoming paper. 


\subsection{Results overview}\label{ssResults}
The main results of this paper are Theorems~\ref{sp} and~\ref{sp_fin} which fully classify the spectral measures of~\eqref{ev} and of~\eqref{exp}. 
Some closely related results were known for $p=1$, and we postpone their review until the end of the section.

For the super-exponential perturbations (\eqref{exp} with $R=\infty$), we establish a sharpening of the spectral characterization as follows. We show that
\begin{equation}\label{GeronimoBaxter}
\limsup_{n\to\infty} \frac{n\log n}{\log (|a_n-a_n^{\circ}| + |{b}_n-b_n^{\circ}|)^{-1}} = \frac{1}{2(g+1)} \limsup_{n\to\infty} \frac{n\log n}{\log |c_n|^{-1}},
\end{equation}
where $c_n$ are the Fourier coefficients of the (suitably modified) a.c. density of the spectral measure of $(a_n,b_n)_{n=1}^\infty$. 
For the free case~\eqref{expFree}, this curious prefactor $\frac{1}{2(g+1)}$ (with $g=0$) on the right-hand side of~\eqref{GeronimoBaxter} has already appeared in Geronimo's result~\cite[Thm 13]{Geronimo}. This equality has a connection to the Baxter theorem~\cite{Bax61} in the theory of orthogonal polynomials on the unit circle, which also relates the asymptotic behaviors of Fourier and recurrence coefficients. We will therefore refer to an equivalence of the type~\eqref{GeronimoBaxter} as a Geronimo--Baxter theorem. Apart from an interest in its own right, we need this result in our investigation of resonances.

Our spectral measures characterization allows us to tackle the associated inverse resonance problem of recovering the operator from its resonances (to be defined in a moment) and eigenvalues. Such a problem appears naturally from the point of view of physics (where resonances of a physical system appear as semi-stable decaying states that are measurable in a laboratory) and has been a topic of active investigation for various classes of operators (including continuous Schr\"{o}dinger operators on the half and full line, one- and two-sided Jacobi operators, CMV operators, etc). 

Mathematically, resonances of an operator can be defined as the poles of the operator's resolvent on the second sheet of a Riemann surface associated with the resolvent set of the operator. As we will see later, this notion makes sense precisely when the perturbation is exponentially decaying~\eqref{exp} or finite range~\eqref{ev}. See Definitions~\ref{s} and~\ref{singularity}, where we make all this rigorous and specific to our setting.

We solve the inverse resonance problem, which consists of reconstruction of the operator (or equivalently, of its spectral measure) out of the knowledge of its resonances and eigenvalues. Moreover,  we find the necessary and sufficient conditions on the configurations of resonances and eigenvalues for such an operator to exist. We will refer to this result as ``existence for the inverse resonance problem''. We also show that such an operator is unique if we restrict ourselves to only finite range perturbations, or, more generally, to the subclass of super-exponential perturbations $(a_n,b_n)_{n=1}^\infty$ that satisfy
\begin{equation}\label{recoverable}
|a_n-a_n^{\circ}| + |{b}_n-b_n^{\circ}| \le C e^{-\eta n \log n}
\end{equation}
for some periodic Jacobi operator $(a_n^{\circ},b_n^{\circ})_{n=1}^\infty$ and fixed constants $\eta>2 (g+1)$, $C>0$. This is essentially sharp: we show that extending the class to operators satisfying~\eqref{recoverable} with $\eta\ge 2(g+1)$ already allows one to find two Jacobi operators with identical eigenvalues and resonances.

After establishing existence and uniqueness, it is natural to raise the question of stability of the inverse resonance problem, especially because of its significance from the point of view of physics.
Indeed, measuring infinitely many states of a system is not feasible, and those finitely many states that get measured will pick up measurement errors. Mathematically, we can think of two Jacobi operators $\calJ=(a_n,b_n)_{n=1}^\infty$ and $\til{\calJ}=(\til{a}_n,\til{b}_n)_{n=1}^\infty$ whose eigenvalues and resonances  in a disk of a large radius $R$ are pairwise $\veps$-close to each other. The resonances of $\calJ$ and $\til{\calJ}$ outside of this disk can be arbitrary. We show that the Jacobi coefficients are then close to each other
$$
|a_n - \til{a}_n | + |b_n - \til{b}_n| \le Q^{n^2} \left(\frac{1}{R^\tau} + \sqrt{\veps}\right) \quad \mbox{ for all } n
$$
with some constants $Q>1$, $0<\tau<1$. We stress that constants $Q$ and $\tau$ do not depend on $\veps$, $R$, or the choice of $\calJ$ and $\tilde\calJ$.
This means that by measuring sufficiently many states ($R\gg 0$) with precise enough measurements ($0<\veps\ll 1$)
we can recover any (finite) number of the Jacobi coefficients to any precision we want.

The organization of the paper is as follows. In Section~\ref{sPrelim} we present the relevant definitions and preliminaries. In Section~\ref{sSpectral} we prove the spectral characterization. In Section~\ref{sMatrix} we obtain a matrix-valued spectral theorem and a Geronimo--Baxter theorem. In Section~\ref{sGB} we establish a Geronimo--Baxter theorem for (scalar) periodic Jacobi matrices.  We solve the inverse resonance problem in Sections~\ref{sResonance} (existence and uniqueness) and~\ref{sResonanceStability} (stability). Finally, in Section~\ref{sApplications} we discuss two corollaries. First is an improvement of the Damanik--Simon~\cite{DS2} theorem which classified Jost functions for~\eqref{evFree} and~\eqref{expFree}. Second is an investigation of the effect that a point mass perturbation of the spectral measure has on the Jacobi coefficients. This was motivated by questions raised by Geronimo in~\cite{Geronimo} for~\eqref{expFree}.




\subsection{Historical discussion}\label{ssHistory}

Theorems~\ref{sp}/\ref{sp_fin} have two directions. For the duration of the paper, $(P_R)\Rightarrow (S_R)$ will be called the ``direct spectral problem'', and $(S_R) \Rightarrow (P_R)$ will be called the ``inverse spectral problem''. Similarly for the direct/inverse resonance problem.

Typically direct problems are much easier to deal with, and our setting is no exception. It was certainly well-known that the spectral measure of a matrix from~\eqref{evFree} or~\eqref{expFree} has a meromorphic a.c. density, no singular-continuous part, and at most finitely many point masses. This was shown by Geronimo~\cite{Geronimo} and Geronimo--Case~\cite{GC}. Similar result for~\eqref{ev} was shown by Geronimo--Van Assche~\cite{GvA86}, and a more complete description can be extracted from the recent Iantchenko--Korotyaev results~\cite{IK3}. The author is not aware of any prior results concerning the direct spectral problem for~\eqref{exp}.

The inverse spectral problem was less understood. For~\eqref{expFree} it was investigated by Geronimo in~\cite{Geronimo}. His two results~\cite[Theorems 13 and 14]{Geronimo} work only under the restriction that the a.c. density of the spectral measure has no poles and finitely many poles, respectively (this restriction was removed by Damanik--Simon in~\cite{DS2}). It also had an implicit condition (positivity of the so-called canonical weights) that we are able to make explicit below (see remarks 2 and 3 after Theorem~\ref{sp} below).

The other two results that should definitely be mentioned are~\cite{DS2} and~\cite{IK3}. Damanik--Simon~\cite{DS2}  fully characterize Jost functions for~\eqref{evFree} and~\eqref{expFree}, and  Iantchenko--Korotyaev~\cite{IK3} do this for~\eqref{ev}. Jost functions and spectral measures are closely connected, so these inverse problems are certainly related to our inverse spectral problem. In fact, the Damanik--Simon ideas (together with the ``Magic'' of Damanik--Killip--Simon~\cite{DKS}) served as the original inspiration and an important stepping stone for our current results. We remark that in the process we are able to establish an improved version of Damanik--Simon theorem, see Theorem~\ref{DamSim} below. The Iantchenko--Korotyaev methods for dealing with eventually periodic case~\eqref{ev} are entirely different, and their results do not seem to imply our 
Theorem~\ref{sp_fin}, nor vice versa. One of the reasons is that the models are not the same: they fix a periodic Jacobi matrix, which is assumed to be known, and
then consider finite range perturbations of it. In our approach for~\eqref{ev}, we fix the support of the spectrum and consider finite range perturbations of any Jacobi matrix from the isospectral
torus, without any other knowledge about it.

The Geronimo--Baxter type Theorem~\ref{baxterSuper} was inspired by Geronimo's result~\cite[Thm 14]{Geronimo} for~\eqref{expFree}, whose precursors were Baxter's~\cite{Bax61}, Geronimo's~\cite{Ger80}, and Geronimo--Nevai's~\cite{GN}.

Questions related to the spectral measures for \textit{block} (matrix-valued) Jacobi operators with exponentially decaying coefficients were studied in another Geronimo's paper~\cite{Geronimo_matrix}, as well as in the author's~\cite{K_jost}.

The direct resonance problem for~\eqref{ev} was completely solved in Iantchenko--Korotyaev~\cite[Thm 1.2]{IK3}\footnote{\cite[Thm 1.2]{IK3} has a mistake: part (2) should not be there}. Their inverse resonance problem assumed additional information. Uniqueness of the inverse resonance problem for super-exponential perturbations of the free case~\eqref{expFree} was solved by Brown--Naboko--Weikard~\cite{BNW_Jac}. Existence for the inverse resonance problem was not solved even for the free case. 

Stability for finite range perturbations of one-sided Jacobi operators was studied by Marletta--Naboko--Shterenberg--Weikard~\cite{MNSW}. The overlap with the stability problem in the current paper is only in the $p=1$ case, where our result is stronger, since we do not restrict ourselves to only finite-range perturbations and our estimate is explicit in $\veps$ and $R$ (in fact with the optimal exponents, at least for $\veps$). We stress however that their method works in impressive generality that in particular allows them to tackle finite range perturbations of \textit{unbounded} Jacobi operators.






We would also like to  mention the Marletta--Weikard~\cite{MW} 
paper that investigates uniqueness and stability of the inverse resonance problem for the discrete Schr\"{o}dinger operators with complex potentials.


Numerous authors also study the direct/inverse resonance problem for other classes of operators. As a non-exhaustive list of papers on this topic, we can suggest \cite{Zw_res} (for the general introduction), \cite{Ble_Jac,IK1,IK2,Kor_Jac} (for perturbations of two-sided Jacobi matrices), \cite{BKW,BW04,Kor_Shr_half,Kor_Stab,MSW10,S_resonances,Zwo01} (for half line Schr\"{o}dinger operators), \cite{Ble_Sch,Chr06,Fir84,Fro97,Hit99,Kor_Shr_real,Kor_Shr_periodic,S_resonances,zworski} (for full line Schr\"{o}dinger operators), 
\cite{SWZ,WZ} (for CMV operators), \cite{BNW_Her} (for perturbations of the Hermite operator).

Though not directly related to our topic of discussion, we would like to mention the results of Volberg--Yuditskii~\cite{VolYud02}, Egorova--Michor--Teschl~\cite{EMT} and Khanmamedov~\cite{Kha06,Kha09}, who study the inverse \textit{scattering} problem for perturbations of periodic (and more general) Jacobi matrices. See also Egorova--Michor--Teschl~\cite{EMT13} and the references therein for more information in this direction.



Finally, the effects of adding and removing a point mass to/from the spectral measures (in the context of Jacobi operators) were initially investigated by Uvarov~\cite{Uva69} and Nevai~\cite{NevBook} (see also Geronimo--Nevai~\cite{GN} and Geronimo~\cite{Geronimo}). Another well-known approach is the double commutation method of Gesztesy--Teschl~\cite{double_comm}.  These results work under great generality, of course, but are not as explicit as we would need for our purposes here. 

\smallskip

A major chunk of technical details has been delegated to the author's previous two works~\cite{K_jost} and~\cite{K_merom}. 
Notation-wise we kept the current paper self-contained. However in the body of several proofs we will have to quote some of the technical results from these two papers.

We should stress that throughout the paper we restrict ourselves to finite gap sets $\fre$ that have ``all gaps open'' (see the discussion in Subsection~\ref{ssPeriodic}). In a sense this is a generic situation for periodic Jacobi matrices. There is little doubt that all the results in the paper should be generalizable not only to the ``closed gaps'' periodic Jacobi matrices, but in fact to all the finite gap Jacobi matrices as well. This is left as an open problem. 


\subsection{Acknowledgements}
It is a pleasure to thank Rowan Killip and Barry Simon for useful discussions. 

The results were completed and presented on conferences in the early 2013 during the author's stay at the Department of Mathematics at UCLA. The write-up was completed during the author's stay at the Department of Mathematics at the Royal Institute of Technology. The author is grateful to both departments for the hospitality.

\section{Preliminaries}\label{sPrelim}

For the textbook presentation of the theory of orthogonal polynomials on the real line (including the spectral theory of periodic Jacobi operators), we refer the reader to the recent Simon's monograph~\cite{Rice}. We follow closely the notation there.

\subsection{Riemann surface $\calS_\fre$}\label{ssRiem}

To a Jacobi operator $\calJ$ and its spectral measure $\mu$~\eqref{spectralDefinition}, we can associate
\begin{equation}\label{m}
m(z)=\int \frac{d\mu(x)}{x-z}, \quad z\notin\esssup\mu,
\end{equation}
the Borel/Stieltjes/Cauchy transform of $\mu$. From~\eqref{spectralDefinition}, $m$ is also the $(1,1)$-entry of the resolvent of $\calJ$. We will refer to this function as the $m$-function of $\calJ$.

For a measure $\mu$ with \eqref{spectrum}, the $m$-function \eqref{m}  is a meromorphic on $\bbC\setminus\fre$ function. It is Herglotz, in the meaning that $\imag m(z)>0$ whenever $\imag z>0$, and $\imag m(z)<0$ whenever $\imag z<0$. As we showed in~\cite{K_merom} (see Lemmas~\ref{merom} and~\ref{merom_fin} below), under \eqref{ev} or \eqref{exp}, the $m$-function has a meromorphic continuation through the bands of $\fre$ to some domain of the second sheet of a certain Riemann surface $\calS_\fre$.

\begin{definition}\label{s}
Assume $\fre$ is a finite gap set \eqref{spectrum}. Define $\calS_\fre$ to be the be the hyperelliptic Riemann surface corresponding to the polynomial $\prod_{j=1}^{g+1}(z-\alpha_j)(z-\beta_j)$.
\end{definition}

We will not give the formal definition, which can be found in many textbooks (see, e.g., \cite[Sect 5.12]{Rice}). Informally $\calS_\fre$ can be described as follows.

Let $\bbC_+ = \{z: \imag z>0\}$, $\bbC_- = \{z:\imag z<0\}$. Denote $\calS_+$ and $\calS_-$ to be two copies of $\bbC\cup\{\infty\}$ with a slit along $\fre$ (include $\fre$ as a top edge and exclude it from the lower), and let $\calS_\fre$ be $\calS_+$ and $\calS_-$ glued together along $\fre$ in the following way: passing from $ \calS_+\cap\bbC_+ $ through $\fre$ takes us to
$ \calS_- \cap\bbC_- $, and from $ \calS_+\cap\bbC_-$ to $ \calS_-\cap\bbC_+$. It is clear that topologically $\calS_\fre$ is  an orientable manifold of genus $g$.

Let $\pi:\calS_\fre\to \bbC\cup\{\infty\}$ be the ``projection map'' which extends the natural inclusions $\calS_+\hookrightarrow \bbC\cup\{\infty\}$, $\calS_-\hookrightarrow \bbC\cup\{\infty\}$.

The following notation will be used frequently throughout the paper.
\begin{definition}\label{sharp}
\hspace*{\fill}
\begin{list}{}
{
\setlength\labelwidth{2em}%
\setlength\labelsep{0.5em}%
\setlength\leftmargin{2.5em}%
}
\item[$\bullet$] For $z\in\bbC\cup\{\infty\}$, denote by $z_+$ and $z_-$ the two preimages $\pi^{-1}(z)$ in $\calS_+$ and $\calS_-$ respectively $($for $z\in\cup_{j=1}^{g+1}\{\alpha_j,\beta_j\}$, $z_+$ and $z_-$ coincide$)$.
\item[$\bullet$] Let $z^\sharp$ be $\Big(\overline{\pi(z)}\Big)_-$ if  $z\in \calS_+\setminus\pi^{-1}(\fre)$, and 
    $\left(\overline{\pi(z)}\right)_+$ if $z\in \calS_-\setminus\pi^{-1}(\fre)$. In order to make this continuous, we make the convention  $z^\sharp=z$ for $z\in \pi^{-1}(\fre)$.
\item[$\bullet$] For a function $m$ on $\calS_\fre$, let  $m^\sharp(z)=\overline{m(z^\sharp)}$.
\end{list}
\end{definition}
\begin{remarks}
1. Note that $\bar{z}^\sharp$ is the point on another sheet (i.e., if $z\in\calS_+$ then $\bar{z}^\sharp\in\calS_-$ and vice versa) satisfying $\pi(z)=\pi(\bar{z}^\sharp)$.

2. Since $m$ in~\eqref{m} is Herglotz, $m(\bar{z}) = \overline{m(z)}$ for $z\in\calS_+$. In \eqref{symmetry} we will see that under~\eqref{ev} or~\eqref{exp}, $m(\bar{z}) = \overline{m(z)}$ for $z\in\calS_-$ as well, so one can think of $m^\sharp(z)$ as $m(\bar{z}^\sharp)$.
\end{remarks}

To simplify the notation, we will commonly write $\calS_+\cap\bbC_+$ instead of $\calS_+\cap\pi^{-1}(\bbC_+)$, etc. 

\subsection{Periodic Jacobi operators and their perturbations}\label{ssPeriodic}

As was mentioned in Section~\ref{sIntro}, the essential spectrum of any $p$-periodic Jacobi matrix is a finite gap set \eqref{spectrum}, where $g\le p-1$. In fact, there exists a polynomial $\Delta$ (called the discriminant) of degree $p$ such that
\begin{equation}\label{discriminant}
\fre = \bigcup_{j=1}^{g+1} [\alpha_j,\beta_j] = \Delta^{-1}([-2,2]).
\end{equation}
In particular $\Delta(z)=z$ when $\fre=[-2,2]$.

If $g = p-1$, then we say that $\calJ$ has all gaps open. It is known that a finite gap set $\fre$ is an essential spectrum of some periodic Jacobi matrix if and only if the harmonic measure of each band is rational. It is an essential spectrum of some periodic Jacobi matrix with all gaps open if and only if the harmonic measure of each band is equal.

It turns out that if there exists at least one periodic Jacobi matrix $\calJ$ with $\sigma_{ess}(\calJ)=\fre$, then there exists a whole set of periodic Jacobi matrices satisfying the same property. In fact, this set is homeomorphic to $(S^1)^g$, a $g$-dimensional torus. See Remark~2 after Thm~\ref{resonanceEv} below and~\cite[Chapt~5]{Rice} for more details. This motivates the following definition.

\begin{definition}\label{torus}
The {isospectral torus} $\calT_\fre$  of  $\fre$ is the set of periodic Jacobi matrices $\calJ$ with $\sigma_{ess}(\calJ)=\fre$.
\end{definition}

From now on assume that all gaps of $\fre$ are open, i.e.,
$$
p=g+1
$$
(this is a generic situation, see \cite[Chapt 5]{Rice}).

Let us define some special subsets $\calS_R$ of $\calS_\fre$ that enter naturally.

\begin{definition}\label{domainE}
Let $x(z)=z+z^{-1}$. For each $R>1$, let
\begin{equation*}\label{calsr}
\calS_R=\calS_+\cup \pi^{-1}(\calE),
\end{equation*}
where $\calE$ is defined to be the union of the interiors of the bounded components of the set $\Delta^{-1}(x(R\,\partial\bbD))$.
\end{definition}

Note that $x(R\,\partial\bbD)$ for various values of $R>1$ are concentric ellipses. Below is an example ($p=3, g=2$) how $\calE$ evolves as $R$ grows. We stress that the set $\calE$ is fully determined by $\fre$ and $R>1$, and in fact, using the results of \cite[Chapt 5]{Rice}, it is easy
to see that $\partial\calE$ are precisely the level sets of the logarithmic potential of the equilibrium measure for $\fre$.


{
\begin{center}
\resizebox{0.4\textwidth}{!}{%
  \includegraphics{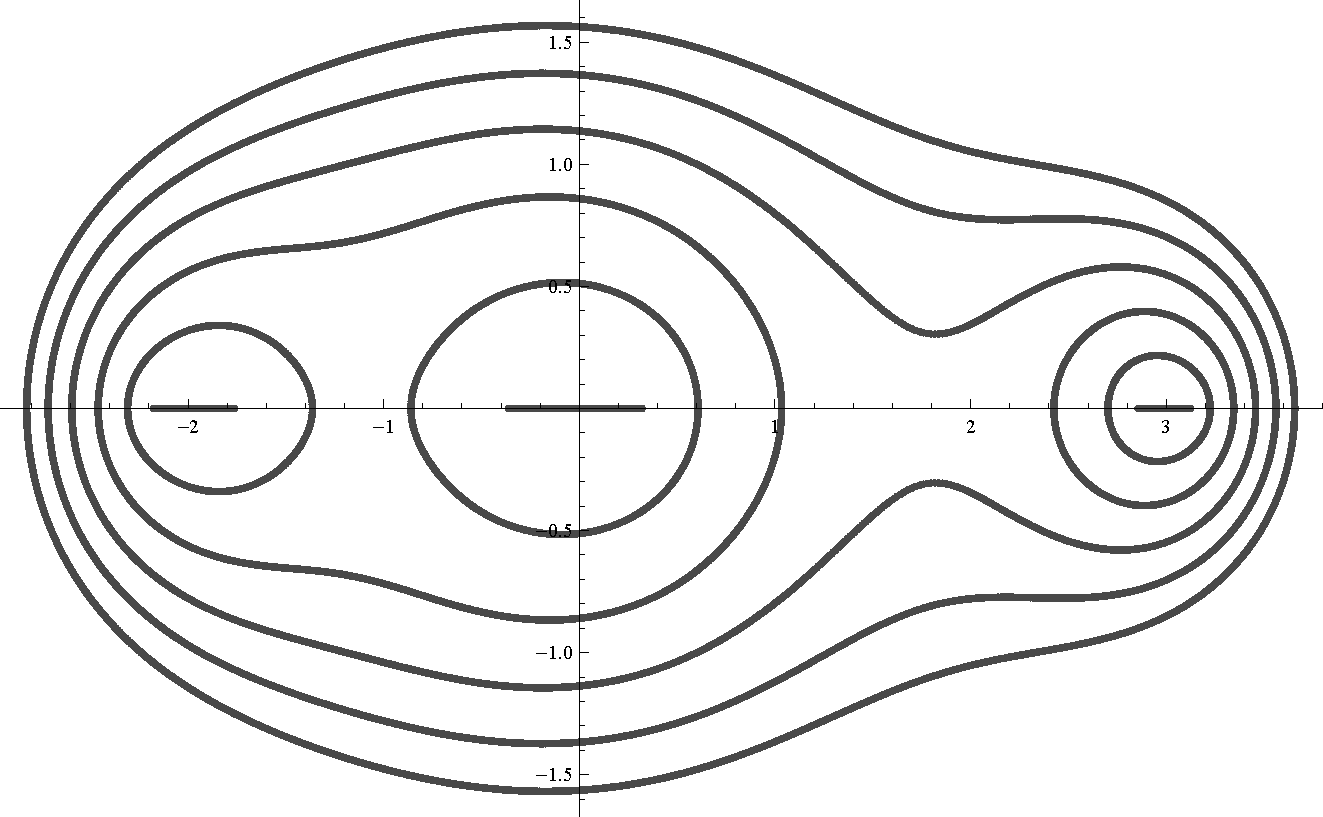}
}
\end{center}
}

\bigskip

The next two results are taken from the author's \cite{K_merom}. We will slightly abuse the notation and use the same symbol $m$ for the $m$-function~\eqref{m} and its meromorphic continuation which is a function on (a subdomain of) $\calS_\fre$.

\begin{lemma}[\cite{K_merom}]\label{merom}
Assume $\esssup\mu=\mathfrak{e}$, and let $m(z)=\int \frac{d\mu(x)}{x-z}$.

The following are equivalent:
\begin{itemize}
\item[($P_R$)] The Jacobi matrix $(a_n,b_n)_{n=1}^\infty$ associated  with $\mu$  satisfies
  \begin{equation*}
  \limsup_{n\to\infty} \left(|a_n-a_n^{\circ}| + |b_n-b_n^{\circ}|\right)^{1/2n}  \le R^{-1},
  \end{equation*}
  where $(a_n^{\circ},b_n^{\circ})_{n=1}^\infty$ is a periodic Jacobi matrix from $\calT_\fre$.
\item[($M_R$)] $m$ satisfies
  \begin{itemize}
  \item[($M_a$)] $m$ has a meromorphic continuation to $\calS_R$; 
  \item[($M_b$)] $m$ has no poles on $\pi^{-1}(\fre)$, except at $\pi^{-1}(\cup_{j=1}^{p}\{\alpha_j,\beta_j\})$, where they are at most simple;
  \item[($M_c$)] $m(z)-m^\sharp(z)$ has no zeros in $\pi^{-1}(\calE)$, except at $\pi^{-1}(\cup_{j=1}^{p}\{\alpha_j,\beta_j\})$, where they are at most simple; 
  \item[($M_d$)] If $m$ has a pole at $z$ for $z\in \pi^{-1}(\calE\setminus \fre)$ then $z^\sharp$ is not a pole of $m$.
  \end{itemize}
\end{itemize}
\end{lemma}

\begin{lemma}[\cite{K_merom}]\label{merom_fin}
Assume $\esssup\mu=\mathfrak{e}$, and let $m(z)=\int \frac{d\mu(x)}{x-z}$.

The following are equivalent:
\begin{itemize}
\item[($P_{f.r.}$)] The Jacobi matrix $(a_n,b_n)_{n=1}^\infty$ associated  with $\mu$  is eventually periodic, i.e., satisfies
  \begin{equation*}
(a_n,b_n)_{n=s}^\infty\in\calT_\fre \quad \mbox{ for  some }s.
  \end{equation*}
\item[($M_{f.r.}$)] $m$ satisfies
  \begin{itemize}
  \item[($M_a'$)] $m$ has a meromorphic continuation to $\calS_\fre$;
  \item[($M_b'$)] $m$ has no poles on $\pi^{-1}(\fre)$, except at $\pi^{-1}(\cup_{j=1}^{p}\{\alpha_j,\beta_j\})$, where they are at most simple;
  \item[($M_c'$)] $m(z)-m^\sharp(z)$ has no zeros in $\calS_\fre\setminus\{\pm\infty\}$, except at $\pi^{-1}(\cup_{j=1}^{p}\{\alpha_j,\beta_j\})$, where they are at most simple;
  \item[($M_d'$)] If $m$ has a pole at $z$ for $z\in \pi^{-1}(\bbC\setminus \fre)$ then $z^\sharp$ is not a pole of $m$.
  \end{itemize}
\end{itemize}
\end{lemma}

Now that we know that $m$ has a meromorphic continuation to the second sheet of $\calS_\fre$, we can define the notion of resonances.

\begin{definition}\label{singularity}
Suppose $m$ satisfies Lemma~\ref{merom} or \ref{merom_fin}, and let $z\in\calS_\fre\setminus\{\infty_+,\infty_-\}$ be a pole of $m$. Then
\begin{itemize}
\item We call $\pi(z)$ a singularity of $\calJ$;
\item If $z\in\calS_+\setminus \pi^{-1}(\fre)$,  
   then we say that $\pi(z)$ is an eigenvalue of $\calJ$;
\item If $z\in\calS_-$, then we say that $\pi(z)$ is a resonance of $\calJ$;
\item If $z\in\calS_-$ and $\imag \pi(z)=0$, then we say that $\pi(z)$ is an anti-bound state of $\calJ$.
\end{itemize}
\end{definition}
\begin{remark}
Eigenvalues correspond to poles in $\calS_+\setminus \pi^{-1}(\fre)$ since (see Lemmas~\ref{merom} and~\ref{merom_fin}) exponentially decaying or finite range perturbations of periodic Jacobi matrices cannot have any point spectrum on $\fre$.
%
\end{remark}

We will also need the Herglotz representation theorem. The matrix-valued version is taken from~\cite[Thm 5.4]{Ges}.

\begin{lemma}\label{herglotz}
Let $\mathfrak{m}$ be a $p\times p$ matrix-valued Herglotz function. Then there exist a $p \times p$ matrix-valued measure $\mu$ on $\bbR$ satisfying $\int_\bbR \frac{1}{1+x^2}d\mu(x)<\infty$, and constant matrices $C=C^*, D\ge\bdnot$ such that
\begin{equation*}\label{herg1}
\mathfrak{m}(z)=C+Dz+\int_\bbR \left(\frac{1}{x-z}-\frac{x}{1+x^2}\right) d\mu(x), \quad z\in \bbC_+.
\end{equation*}
The absolutely continuous part of $\mu$ can be recovered from this representation by
\begin{equation*}\label{herg2}
f(x)\equiv\frac{d\mu}{dx}=\pi^{-1}\lim_{\veps\downarrow0} \imag \mathfrak{m}(x+i \veps),
\end{equation*}
and the pure point part by
\begin{equation*}\label{herg3}
\mu(\{\lambda\})=\lim_{\veps \downarrow 0} \veps\, \imag \mathfrak{m}(\lambda+i\veps) = \tfrac{1}{i} \lim_{\veps\downarrow 0} \veps\, \mathfrak{m}(\lambda+i\veps).
\end{equation*}
\end{lemma}

Because we will be discussing meromorphic continuations, the next lemma will prove to be useful.

\begin{lemma}[\cite{Greenstein}]\label{greenstein}
Let $m$ be as in \eqref{m}. Then $m$ can be analytically continued from $\calS_+ \cap \bbC_+$ through an interval $I\subset\bbR$ if and only if the associated measure $\mu$ is purely absolutely continuous on $I$, and the density $f(x)=\frac{d\mu}{dx}$ is real-analytic on $I$. In this case, the analytic continuation of $m$ into some domain $\mathcal{D}_-$ of $\calS_-\cap \bbC_-$ is given by
\begin{equation*}
m(z_-)=\overline{m(\bar{z}_+)}+2\pi i f(z), \quad z\in\pi(\mathcal{D}_-),
\end{equation*}
where $f(z)$ is the complex-analytic continuation of $f$ to $\mathcal{D}_-$.
\end{lemma}


\section{Spectral measures}\label{spectr}\label{sSpectral}
As we are about to see, locations of the eigenvalues of Jacobi operators from~\eqref{ev} and~\eqref{exp} are required to satisfy a certain property with respect to the locations of the anti-bound states. Loosely speaking, every even-numbered real singularity (when counted starting from any of the edges of $\fre$ in the direction away from the band), cannot be an eigenvalue and therefore must be an anti-bound state. For a lack of a better term we will call it the ``oddly interlacing'' property. Note that in particular it implies (but is stronger than) the following statement: between any two consecutive eigenvalues (which are located in the same simply-connected component of $\mathcal{E}$ and are not separated by a band) there is an odd number of anti-bound states (counted according to their multiplicities).

\begin{definition}\label{OI_light}
$($Oddly interlacing property$)$ Let $\fre$ be a finite gap set
\begin{equation*}
\fre = \bigcup_{j=1}^{p} [\alpha_j,\beta_j], \quad \alpha_j<\beta_j <\alpha_{j+1},
\end{equation*}
and $\mathcal{E}$ be an open set in $\bbC$ containing $\fre$. Suppose we are given two sets $($repeated according to their multiplicities$)$ of real points: $\{e_j\}_{j=1}^N$ $(N<\infty)$ and $\{r_j\}_{j=1}^K$ $(K\le\infty)$. Denote $\{s_j\} = \{e_j\}_{j=1}^N \cup \{r_j\}_{j=1}^K  $ $($with multiplicities preserved$)$. We will say that
$$
\{e_j\}_{j=1}^N \mbox{ oddly interlace with } \{r_j\}_{j=1}^K \mbox{ on the set } \mathcal{E}
$$
if $\{e_j\}_{j=1}^N \cap \{r_j\}_{j=1}^K \cap \mathcal{E} = \varnothing$ and for all $k$, $1\le k \le p$, the following holds:
\begin{itemize}
\item Suppose $\delta>0$ is the largest number such that $(\beta_k,\beta_k+\delta)\subset\mathcal{E}\setminus\fre$, 
and let $[\beta_k,\beta_k+\delta)\cap \{s_j\}=\{x_j\}$ $($with multiplicities preserved$)$, where
    \begin{equation}\label{interlacing_light}
    \beta_k\le x_1\le x_2 \le x_3 \le \ldots .
    \end{equation}
     Then $\{x_2,x_4,\ldots,\} \cap \{e_j\}_{j=1}^N = \varnothing$.

\item Suppose $\delta>0$ is the largest number such that $(\alpha_k-\delta,\alpha_k)\subset\mathcal{E}\setminus\fre$, 
and let $(\alpha_k-\delta,\alpha_k]\cap \{s_j\}=\{y_j\}$ $($with multiplicities preserved$)$, where $$\ldots\le y_3\le y_2 \le y_1 \le \alpha_k.$$ Then $\{y_2,y_4,\ldots,\} \cap \{e_j\}_{j=1}^N = \varnothing$;
\end{itemize}
\end{definition}
\begin{remark}
If $N=0$ (no eigenvalues), then this property trivially holds for any configuration of $\{r_j\}$.
\end{remark}

Now we can state the characterization of the spectral measures. Let us define
\begin{equation*}
\sg_\fre(x) = \begin{cases}
(-1)^{p-k} & \mbox{ if } x\in (\alpha_k,\beta_k), \\
0 & \mbox{ otherwise},
\end{cases}
\end{equation*}
a function that changes sign from one band to another. Let
\begin{equation}\label{r}
r(x)=\prod_{j=1}^p (x-\alpha_j)(x-\beta_j).
\end{equation}

\begin{theorem}\label{sp}
Let $R>1$ and $\calE$ be as in Definition~\ref{domainE}. The following are equivalent:
\begin{itemize}
\item[$(P_R)$] Jacobi matrix $\calJ\equiv (a_n,b_n)_{n=1}^\infty$  satisfies
  \begin{equation*}
  \limsup_{n\to\infty} \left(|a_n-a_n^{\circ}| + |b_n-b_n^{\circ}|\right)^{1/2n}  \le R^{-1},
  \end{equation*}
  where $(a_n^{\circ},b_n^{\circ})_{n=1}^\infty$ is a periodic Jacobi matrix from $\calT_\fre$.
\item[$(S_R)$] Spectral measure $\mu$ of $\calJ$ is of the form
\begin{equation}\label{meas}
d\mu(x) = \frac{\sqrt{|r(x)|}}{|a(x)|} 1_{x\in\fre} dx + \sum_{j=1}^N w_j \delta_{E_j},
\end{equation}
where
\begin{itemize}
   \item[$(S_a)$] $a(z)$ is a real-analytic function in $\calE$ which satisfies $\sgn \,a(x) = \sg_\fre(x)$ on $\operatorname{Int}(\fre)$;
  \item[$(S_b)$]  $N<\infty$ and $E_j\in \bbR\setminus\fre$. Each $E_j$ that belongs to $\calE$ is a simple zero of $a(z)$. Moreover, 
  $\{E_j\}_{j=1}^N$  oddly interlace $($Def.~\ref{OI_light}$)$ with $($anti-bound states$)$ $\{R_j\}_{j=1}^{K}$ on the set $\calE$. Here $\{R_j\}_{j=1}^K $ are defined to be the zeros of $a(z)$ in $\bbR\setminus\{E_j\}_{j=1}^N$, repeated according to their multiplicities; 
  \item[$(S_c)$] For each  $1\le j \le N$, if $E_j\in \calE$ then
  \begin{equation}\label{weights}
  w_j= 2\pi \Big| \res_{z=E_j}\frac{\sqrt{r(z)}}{a(z)} \Big|.
  \end{equation}
\end{itemize}
\end{itemize}
\end{theorem}

\begin{remarks}
1. Note that Definition~\ref{singularity} allows the situation when a point is a resonance and an eigenvalue simultaneously. However, as seen in Lemma~\ref{merom} and~\ref{merom_fin}, this can never happen in $\calE$. 
As is clear from the proof, in this case resonances occur precisely at those zeros of the denominator of $a(z)$ that are not eigenvalues. This means that points $\{R_j\}_{j=1}^K$ defined in $(S_b)$ are anti-bound states from Definition~\ref{singularity}.

2. We stress that $(S_b)$ is a statement about which points are allowed to be eigenvalues. There is no implicit restriction on $a(z)$ here, and any function $a(z)$ that satisfies ($S_a$) (up to a multiplicative normalization constant) can occur in~\eqref{meas}.

3. Similarly, $(S_c)$ is a statement about eigenweights only (again, up to an inconsequential normalization). Indeed, note that each $w_j$ in~\eqref{weights} is positive, so there is no implicit positivity restriction here either (unlike the canonical weights in~\cite{Geronimo} or~\cite{DS2}). 

\end{remarks}

\begin{theorem}\label{sp_fin}
The following are equivalent:
\begin{itemize}
\item[$(P_{f.r.})$] Jacobi matrix $\calJ\equiv (a_n,b_n)_{n=1}^\infty$ is eventually periodic, i.e., satisfies
  \begin{equation*}
(a_n,b_n)_{n=s}^\infty\in\calT_\fre \quad \mbox{ for  some }s.
  \end{equation*}
\item[$(S_{f.r.})$] Spectral measure $\mu$ of $\calJ$ is of the form
\begin{equation}\label{meas2}
d\mu(x) = \frac{\sqrt{|r(x)|}}{|a(x)|} 1_{x\in\fre} dx + \sum_{j=1}^N w_j \delta_{E_j},
\end{equation}
where
\begin{itemize}
\item[$(S'_a)$] $a(z)$ is a real polynomial which satisfies $\sgn \,a(x) = \sg_\fre(x)$ on $\operatorname{Int}(\fre)$; 
  \item[$(S'_b)$]  $N<\infty$ and $E_j\in \bbR\setminus\fre$. Each $E_j$ is a simple zero of $a(z)$. Moreover, 
  $\{E_j\}_{j=1}^N$  oddly interlace $($Def.~\ref{OI_light}$)$ with $($anti-bound states$)$ $\{R_j\}_{j=1}^{K}$ on the set $\bbC$. Here $\{R_j\}_{j=1}^K $ are defined to be the zeros of $a(z)$ in $\bbR\setminus\{E_j\}_{j=1}^N$, repeated according to their multiplicities;
  \item[$(S'_c)$] For each  $1\le j \le N$, 
  \begin{equation}\label{weights2}
  w_j= 2\pi \Big| \res_{z=E_j}\frac{\sqrt{r(z)}}{a(z)} \Big|.
  \end{equation}
\end{itemize}
\end{itemize}
\end{theorem}

\begin{remarks}
1. It may be tempting to take $a$ to be a constant function $a(z)=A$ in a hope to get an eventually periodic Jacobi matrix. A constant function is indeed entire, but it does not satisfy the sign condition in $(S'_a)$ (unless $p=1$). In fact, $d\mu(x)=\tfrac1A \sqrt{|r(x)|} dx$  does not even correspond to a super-exponential decay $R=\infty$. Indeed, going back to Theorem~\ref{sp}, one can see that we get \textit{finite} rate of exponential decay equal to the largest $R$ for which $\calE$ consists of $p$ \textit{disjoint} pieces. Indeed, if $\calE$ contains $p$ disjoint pieces, then one can just redefine signs of $\pm A$ on each of the bands, and Theorem~\ref{sp} can be applied.

2. In fact, $(S'_a)$ implies that there must be at least $p-1$ singularities.

3. Again, $(S'_b)$ and $(S'_c)$ do not impose any additional restrictions on $a$. So up to a multiplicative normalization constant, any polynomial $a(z)$ that satisfies the sign condition in $(S'_a)$ and keeps~\eqref{meas2} integrable (that amounts to no zeros on $\Int(\fre)$ and at most simple ones at the edges) can occur as the denominator of the spectral measure~\eqref{meas2}.
\end{remarks}

\begin{proof}[Proof of Theorem~\ref{sp}]
For the duration of the proof, let us choose the branch of the square root for the complex function $\sqrt{r(z)}$ (defined on $\bbC\setminus\fre$) that is positive on $(\beta_p,+\infty)$. 
Note that this function changes sign from one gap to another. 
Using this, one can see that \textit{assuming $(S_a)$ and $(S_b)$}, the expression in~\eqref{weights} can be equivalently written as
\begin{equation}\label{weights_complex}
  w_j=-2\pi\res_{z=E_j}\frac{\sqrt{r(z)}}{a(z)}.
\end{equation}
Indeed, let us show that in this case each~\eqref{weights_complex} is automatically positive.  Suppose  that $E_j\in(\beta_p,+\infty)\cap \calE$. By $(S_b)$ $E_j$ is a pole of $a(z)$ of order $1$, i.e.,~\eqref{weights_complex} is $-2\pi \tfrac{\sqrt{r(E_j)}}{a'(E_j)}$. By $(S_a)$, $a(z)$ is positive on $(\alpha_p,\beta_p)$, and by the oddly interlacing property, there is an even number of  zeros of $a(z)$ (counting with multiplicities) on the interval $[\beta_p,E_j)$ . Thus $a'(E_j)<0$, and since $\sqrt{r(E_j)}>0$ is positive on $(\beta_p,+\infty)$, we conclude that $w_j>0$. The argument for $E_j$'s on any of the gaps or on $(-\infty,\alpha_1)$ are similar if one uses the sign condition on $a(z)$ from $(S_a)$ and the sign changes of $\sqrt{r(z)}$.

Let us also define the function  $\sqrt{\til{r}(z)}:\calS_\fre \to \bbC$ that coincides with $\sqrt{r(z)}$ on $\calS_+$, and equals to its negative on $\calS_-$. That is, $\sqrt{\til{r}(\bar{z}^\sharp)}=-\sqrt{\til{r}(z)}$. It is not hard to see that this function is in fact meromorphic on the whole surface $\calS_\fre$.

\smallskip

$(P_R) \Rightarrow (S_R)$ We apply Lemmas~\ref{merom} and \ref{herglotz}. $(P_R)$ implies $(M_R)$. By ($M_a$) and ($M_b$) we can see that $\mu(x)$ consists of an absolutely continuous part $f(x)dx$ on $\fre$ and finitely many point masses outside of $\fre$ (finiteness comes from the fact that poles of a meromorphic function $m$ cannot have an accumulation point in the interior of the domain of meromorphicity).



By $(M_a)$ $m$ has a meromorphic continuation to $\pi^{-1}(\calE)$. Let us first apply Lemma~\ref{greenstein} to $m$ on the domain of $[(\calS_+\cap\bbC_+) \cup (\calS_-\cap\bbC_-)] \cap \pi^{-1}(\calE)$. 
We obtain that $f(x)$ has a complex-meromorphic continuation $f(z)$ to $(\calE \setminus \bbR)\cup\fre$, and
\begin{equation}\label{secondSheetM}
m(z_-) = \overline{m(\bar{z}_+)} + 2\pi i f(z), \quad z\in \bbC_-\cap \calE.
\end{equation}
Note that $f(z)$ is a function of a complex variable $z\in\bbC$.

Let us now apply Lemma~\ref{greenstein} on the domain of $[(\calS_+\cap\bbC_-) \cup (\calS_-\cap\bbC_+)] \cap \pi^{-1}(\calE)$. Note that $\imag m((x-i0)_+) = -\pi f(x)$, which means that
\begin{equation}\label{secondSheetM2}
m(z_+)= \overline{m(\bar{z}_-)} - 2\pi i f(z), \quad z\in\bbC_-\cap  \calE.
\end{equation}

Now using $\overline{m(\bar{z}_+)}=m(z_+)$, we obtain
\begin{equation}\label{symmetry}
\overline{m(z_-)}=m(\bar{z}_-),  \quad z\in \calE. 
\end{equation}
In particular (see Definition~\ref{sharp}),
\begin{equation}\label{msharp}
m^\sharp(z)=m(\bar{z}^\sharp),  \quad z\in \calS_R. 
\end{equation}

Let us lift $f(z)$ to a function $\tilde{f}(z): \calS_\fre \to \bbC$ as follows. We define
\begin{align*}
\tilde{f}(z) = f(\pi(z)), \quad &\mbox{if } z\in \calS_+\cap\bbC_+ \cap \pi^{-1}(\calE), \\
\tilde{f}(z) = f(\pi(z)), \quad &\mbox{if } z\in \calS_-\cap\bbC_- \cap \pi^{-1}(\calE), \\
\tilde{f}(z) = -f(\pi(z)), \quad &\mbox{if } z\in \calS_+\cap\bbC_- \cap \pi^{-1}(\calE), \\
\tilde{f}(z) = -f(\pi(z)), \quad &\mbox{if } z\in \calS_-\cap\bbC_+ \cap \pi^{-1}(\calE).
\end{align*}

Then~\eqref{secondSheetM} and~\eqref{secondSheetM2} imply that
\begin{equation}\label{tildeF}
\tilde{f}(z)=\tfrac{1}{2\pi i}(m(z)-m^\sharp(z)), \quad z\in \pi^{-1}(\calE)
\end{equation}
holds, and so $\tilde{f}$ is meromorphic on $\pi^{-1}(\calE)$.


 By \eqref{msharp} we also have $\tilde{f}(\bar{z}^\sharp) = -\tilde{f}(z)$, and therefore
$$
\frac{\tilde{f}(\bar{z}^\sharp)}{\sqrt{\til{r}(\bar{z}^\sharp)}} = \frac{\tilde{f}(z)}{\sqrt{\til{r}(z)}}.
$$
Thus $\frac{\tilde{f}(z)}{\sqrt{\til{r}(z)}}$ is a function of $\pi(z)$, so it can be viewed as a function of \textit{complex} variable (that is, $\calE\to\bbC$ rather than $\pi^{-1}(\calE) \to \bbC$). We claim it is meromorphic on $\calE$. This is clear for all points except for the branch points. Observe that around a branch point, say $z=\beta_k$, we have Taylor's expansion
\begin{equation}\label{taylor}
m(z)=\sum_{j=-1}^\infty k_j (z-\beta_k)^{j/2}
\end{equation}
Indeed, recall that local coordinates at a branch point $\beta_k$ are given in terms of $(z-\beta_k)^{1/2}$, not $z-\beta_k$; and that $m$ has at most order $1$ pole at $\beta_k$ by ($M_b$).
Then $m^\sharp(z)=\sum_{j=-1}^\infty (-1)^{j} k_j (z-\beta_k)^{j/2}$, and therefore
$$
\tilde{f}(z)=\tfrac{1}{2\pi i}(m(z)-m^\sharp(z)) = \tfrac{1}{\pi i}\sum_{j=-1}^\infty k_{2j+1}(z-\beta_k)^{j+1/2}.
$$
Since $\sqrt{\til{r}(z)}^{-1} = (z-\beta_k)^{-1/2} \sum_{j=0}^\infty \hat{k}_j (z-\beta_k)^j$ for some constants $\hat{k}_{2j+1}$, we obtain
$$
\frac{\tilde{f}(z)}{\sqrt{\til{r}(z)}}= \frac{1}{\pi i} \sum_{j=-1}^\infty \hat{\hat{k}}_{2j+1} (z-\beta_k)^{j},
$$
which has no branching at the point $\beta_k$. Therefore $\frac{\tilde{f}(z)}{\sqrt{\til{r}(z)}}$  can be viewed as a meromorphic function on $\calE$ {in $z\in\bbC$ variable}, which we will denote by $\frac{1}{ia(z)}$ for convenience. Therefore
$$
\tilde{f}(z)=\frac{\sqrt{\til{r}(z)}}{ia(\pi(z))}.
$$
Carefully examining the signs, one can see that on $\pi^{-1}(\fre)\cap \calS_+$ the last expression is equal to $f(x)=\frac{\sqrt{|r(x)|}}{a(x)}\sg_\fre(x)$.

Since the density $f(x)$ is a nonnegative function on $\fre$, we conclude that $a(z)$ is real for $z\in\bbR$ and satisfies $\sgn \,a(x) = \sg_\fre(x)$ on $\fre$.

In fact, $a(z)$ is \textit{analytic} on $\calE$, since we know by ($M_c$) that $\til{f}(z)$ has no zeros on $\calE$ except possibly for the band edges where the zeros are at most simple.

This establishes $(S_a)$.

To see $(S_c)$, suppose that $E_j\in \calE$ is an eigenvalue of $\calJ$. Let us put $z=(E_j)_+$ in~\eqref{tildeF}, and take residues of both sides. Since $(E_j)_-\in\calS_-$ cannot be a pole of $m$ by ($M_d$), the residue of the right-hand side is $\tfrac{1}{2\pi i} (-w_j -0)$, while the residue of the left-hand side is 
$$
\tfrac{1}{i} \res_{z=E_j} \frac{\sqrt{r(z)}}{a(z)},
$$
which shows that $E_j$ is a simple zero of $a$ and that~\eqref{weights_complex} holds. By the discussion in the beginning of the proof, this will establish $(S_c)$ if we can show $(S_b)$.


Note that every eigenvalue $E_j\in \bbR\setminus \fre$ by ($M_b$). To show the oddly interlacing property, let us order the real singularities $\{x_j\}$  of $\calJ$ on $[\beta_k,\beta_k+\delta)\subset \calE$ as in~\eqref{interlacing_light}.
We need to show that $x_j$ is a resonance for every even $j$.

Just as before, let the Taylor's series of $m$ at $\beta_k$ be~\eqref{taylor}, so that
\begin{equation}\label{temp1}
m(z)-m^\sharp(z) =2k_{-1}(z-\beta_k)^{-1/2}+ 2k_1 (z-\beta_k)^{1/2} + O((z-\beta_k)^{3/2}).
\end{equation}
By~\eqref{tildeF},~\eqref{temp1} belongs to $i\bbR_+$ on $\pi^{-1}((\alpha_k,\beta_k))\cap \calS_+$. Choosing for definiteness $(z-\beta_k)^{1/2}$ to be positive for $z>\beta_k$, we can see that this forces
 either $k_{-1}<0$, or $k_{-1}=0$ and $k_1>0$ (note that it is not possible to have $k_{-1}=k_1=0$ because of \eqref{temp1} and the fact that zeros of $m-m^\sharp$ at the band edges are at most simple by ($M_c$)).

Now, if $k_{-1}<0$, then $m-m^\sharp$ is negative to the right of $\beta_k$. Since in this case $\beta_k$ is a first order pole, $x_1=\beta_k$. Note that $m-m^\sharp$ cannot be equal to zero by ($M_c$) on $(\beta_j,\alpha_{j+1})$, and thus
\begin{equation}\label{residue}
\lim_{z\to (x_2-0)_+} m(z)-m^\sharp(z)=-\infty
\end{equation}
This implies that $x_2$ must be a resonance: indeed, if $x_2$ were an eigenvalue, then $\lim_{z\to (x_2-0)_+} m(z)=+\infty$ (since $\lim_{x\to x_2-0} \tfrac{1}{x_2-x} = +\infty$) and $\lim_{z\to (x_2-0)_+} m^\sharp(z)$ is finite by ($M_d$), which would contradict to \eqref{residue}.

Now, if $k_{-1}=0$ and $k_1>0$, then $m-m^\sharp$ is positive to the right of $\beta_k$. Since $m-m^\sharp$ cannot be equal to zero, we obtain $$\lim_{z\to (x_1-0)_+} m(z)-m^\sharp(z)=+\infty.$$
If $x_1=x_2$ then it is a resonance, since $m$ is Herglotz on $\calS_+$ and therefore cannot have second order poles there. If $x_1\ne x_2$, then using non-vanishing of $m-m^\sharp$ again,
\begin{align*}
&\lim_{z\to (x_1+0)_+} m(z)-m^\sharp(z)=-\infty,\\
&\lim_{z\to (x_2-0)_+} m(z)-m^\sharp(z)=-\infty,
\end{align*}
which implies that $x_2$ is a resonance by the same arguments as above. Checking the signs of $m-m^\sharp$ further, one sees that~\eqref{residue} holds at any $x_j$ with even $j$, which means they are resonances.

The arguments for $y_j$'s are analogous.

\bigskip

$(S_R) \Rightarrow (P_R)$ Assume $\mu$ is \eqref{meas} and satisfies $(S_a)$, $(S_b)$, $(S_c)$.


Now we want to show that the $m$-function \eqref{m} of $\mu$ satisfies conditions ($M_R$) of Lemma~\ref{merom}.

On $\fre$ the absolutely-continuous density $f(x)=\tfrac{d\mu(x)}{dx}$ of $\mu$ coincides with the function $\frac{\sqrt{r(z)}}{i a(z)}$, which, as above, can be  lifted to a meromorphic function $\tilde{f}$ on $\calS_R$. Now we can define $m^\sharp$ via~\eqref{tildeF} and then using \eqref{secondSheetM}--\eqref{secondSheetM2} see that $m^\sharp$ is indeed the meromorphic continuation of $m$.
This proves ($M_a$). ($M_b$) follows Lemma~\ref{herglotz} and $E_j\notin\fre$.
($M_c$) follows from $m-m^\sharp = 2\pi i \tilde{f}$ and analyticity of $a$. Finally, if $E_j \in \calE_R\setminus\fre$ is an eigenvalue, then using~\eqref{tildeF} and~\eqref{weights_complex}, we obtain that the residue of $m$ at $z=(E_j)_-$ is zero. This proves ($M_d$) and completes the proof.
\end{proof}

\begin{proof}[Proof of Theorem~\ref{sp_fin}]
This can be proved by using the same arguments but applying Lemma~\ref{merom_fin} instead of \ref{merom}. In the process one should note that $m$ being meromorphic on the whole surface $\calS_\fre$ implies that $a(z)$ is entire and has a finite order pole at $\infty$, i.e., $a(z)$ is a polynomial.
\end{proof}

\section{Matrix-valued problem}\label{sMatrix}

For technical reasons we need to introduce some notation and derive one result (namely a Geronimo--Baxter theorem) from the spectral theory of \textit{block} Jacobi matrices. 
A reader interested in a more detailed view at this theory should see Damanik--Pushnitski--Simon~\cite{DPS} or~\cite[Chapt~4]{Rice}.

By a block Jacobi matrix/operator we call a Hermitian operator on $\ell^2(\bbZ_+)$ of the form
\begin{equation*}\label{blockJacobi}
\calJ=\left(
\begin{array}{cccc}
B_1&A_1&{\bdnot}&\\
A^*_1 &B_2&A_2&\ddots\\
{\bdnot}&A^*_2 &B_3&\ddots\\
&\ddots&\ddots&\ddots\end{array}\right),
\end{equation*}
where $B_j^*=B_j$. Here each of $A_j$ and $B_j$ is a $p\times p$ matrix. We will use the notation $(A_n,B_n)_{n=1}^\infty$ for such an operator. Its spectral measure is the $p\times p$ matrix-valued Hermitian probability  measure $\mu(x)$ whose $(j,k)$-entry $\mu_{jk}(x)$ ($1\le j,k\le p$) is defined via
\begin{equation*}\label{matrixSpectralDefinition}
\int_\bbR f(x)d\mu_{jk}(x) = \langle \delta_j, f(\calJ) \delta_k \rangle,
\end{equation*}
where $\delta_k$ is the vector having $1$ on the $k$-th position and $0$ everywhere else.

If $\calJ=(A_n,B_n)_{n=1}^\infty$ is a block Jacobi matrix, then for any integer $s\ge 0$ we define $\calJ^{(s)}$ to be $(A_{n+s},B_{n+s})_{n=1}^\infty$, that is, the block Jacobi matrix obtained from $\calJ$ by removing the first $s$ of the matrix-valued rows and corresponding columns (i.e., $sp$ scalar rows and columns).

\subsection{Geronimo--Baxter theorem for block Jacobi operators}
The equivalence $(1)\Leftrightarrow (2)$ of the next theorem is the matrix-valued analogue of $(1)\Leftrightarrow (3)$ of Theorem~\ref{baxterSuper} below, which is why we call this a Geronimo--Baxter theorem.


\begin{theorem}\label{baxterMatrix}
Suppose that
\begin{equation}\label{blockSuperExp}
\lim_{n\to\infty} (||A_n A_n^*-\bdone|| + ||B_n||)^{1/2n} = 0.
\end{equation}
Then the a.c. density of the matrix-valued spectral measure $\mu_\Delta$ of $\calJ$ is
\begin{equation}\label{acPart}
\frac{d \mu_\Delta}{dx} = \mathfrak{a}(x)^{-1} \sqrt{4-x^2} \, 1_{x\in[-2,2]},
\end{equation}
where $\mathfrak{a}(z)$ is an entire function. Moreover, the following are equivalent:
\begin{itemize}
\item[(1)]
\begin{equation*}\label{growthOrder}
\limsup_{r\to\infty} \frac{\log\log \sup_{|z|=r} ||\mathfrak{a}(z)||}{\log r} =\rho ;
\end{equation*}
\item[(2)] The Jacobi coefficients $(A_n,B_n)_{n=1}^\infty$ satisfy
$$
\limsup_{n\to\infty} \frac{n\log n}{\log (||A_n A_n^*-\bdone|| + ||B_n||)^{-1}} = \frac{\rho}{2}.
$$
\end{itemize}
\end{theorem}

\begin{remark}
 In fact the condition~\eqref{blockSuperExp} implies not only~\eqref{acPart}, but also guarantees that the spectral measure contains no singular-continuous part, and only finitely many point masses each of which is canonical. More details are in~\cite[Thms~3.7,~3.8]{K_jost}. One can also deduce a result about  matrix-valued spectral measures of block Jacobi matrices satisfying $\limsup_{n\to\infty} (||A_n A_n^*-\bdone|| + ||B_n||)^{1/2n} \le R^{-1}$ for $1<R<\infty$.
\end{remark}
\begin{proof}
This theorem is the refinement of the argument from the author's~\cite{K_jost}, which in turn uses the ideas of  Damanik--Simon~\cite{DS2}. First of all, under the condition~\eqref{blockSuperExp} we can repeat the arguments in the beginning of the proof of Theorems~\ref{sp} but using the  matrix-valued $m$-function criterion~\cite[Thm~3.8]{K_jost} instead of the scalar Lemma~\ref{merom}. This implies that the a.c. density of $\mu_\Delta$ is indeed of the form~\eqref{acPart} with $\mathfrak{a}$ being an entire function. Also note that by \cite[Theorem~4.6(vi)]{K_jost},
\begin{equation}\label{funcA}
\mathfrak{a}(z+z^{-1}) = 2\pi u(\bar{z}^{-1})^* u(z),
\end{equation}
where $u(z)$ is the Jost function of $\calJ$ (for the definition and properties see~\cite{K_jost}). Since $u(0)=\bdone$, the behavior of $||\mathfrak{a}||$ and $||u||$ for large $|z|$ are identical. Let us now show the ``moreover'' part of the theorem.

Suppose first that for some constant $K_0>0$,
\begin{equation}\label{superexp}
||A_n A_n^*-\bdone || + ||B_n|| \le K_0 e^{-\frac{2}{\rho+\varepsilon} n\log n}, 
\end{equation}
where we fixed some $\varepsilon>0$.

Following the proof of \cite[Lemma~4.3]{K_jost} and using \eqref{superexp}, we can see that there exists some constant $K_1>0$ such that
\begin{equation}\label{refer1}
||g_n(z)||\le K_1 |z|^{2n}, \qquad ||c_n(z)||\le K_1 |z|^{2n}
\end{equation}
for all $|z|>1$ (see \cite[Eq.~(4.15)]{K_jost}). Here $g_n(z)$ and $c_n(z)$ are the functions defined by the Geronimo--Case recursions \cite[eq. (4.6)--(4.7)]{K_jost}.
Now we use \cite[eq. (4.16)]{K_jost} to see that
\begin{equation}\label{refer2}
||g_{n+1}(z)-g_{n}(z) || \le K_2 (||A_n A_n^*-\bdone || + ||B_n||)|z|^{2n+2}
\end{equation}
for another constant $K_2>0$ (note that \cite[eq. (4.16)]{K_jost} has a typo: instead of $\left[ \max(1,r) \right]^{2n}$ there should be $\left[ \max(1,r) \right]^{2}$).
Thus
\begin{multline}\label{normU}
||u(z)||=||\lim_{n\to\infty} g_n(z) || \le 1+\sum_{n=0}^\infty ||g_{n+1}(z)-g_n(z)|| \\
\le  K_3\, |z|^2 \sum_{n=0}^\infty |z|^{2n} e^{-\frac{2}{\rho+\varepsilon} n\log n}.
\end{multline}
Now note that the order of the exponential growth of the entire function $t(z):= z \sum_{n=0}^\infty z^{n} e^{-\frac{2}{\rho+\varepsilon} n\log n}$ is
$$
\limsup_{n\to\infty} \frac{n\log n}{\log e^{\frac{2}{\rho+\varepsilon} n\log n} }   = \frac{\rho+\varepsilon}{2}.
$$
Thus
$$
||u(z)|| \le t(|z|^2) \le e^{|z|^{\rho+2\varepsilon}} \qquad \mbox{for large } |z|.
$$

Conversely, suppose that
\begin{equation}\label{converse1}
||u(z)|| \le e^{|z|^{\rho+\varepsilon}} \qquad \mbox{for large } |z|.
\end{equation}
Let us use the notation from \cite{K_jost}: define $u^{(n)}$ to be the Jost function of $\calJ^{(n)}$, and take an arbitrary $1<R_1<\infty$.  For an entire matrix-valued function $f$ and any $R$ we define
$$
|||f|||_{R}:=\left( \int_{-\pi}^\pi ||f(R e^{i\theta})-f(0)||^2 \frac{d\theta}{2\pi} \right)^{1/2}.
$$
Now we invoke \cite[eq. (5.11)]{K_jost} (together with the inline formula that follows it in~\cite{K_jost}):
$$
|||u^{(n+1)}|||_{R_1} \le K_4 \, R_1^{-2n} |||u|||_{R_1} (1+\varepsilon)^{2n}.
$$
Let us choose $R_1$ to be $n$-dependent: $R_1=n^{1/(\rho+\varepsilon)}$.  Then~\eqref{converse1} becomes
$$|||u|||_{n^{1/(\rho+\varepsilon)}} \le K_5 e^{R_1^{\rho+\varepsilon}} = K_5 e^n,$$
 and so we get
$$
|||u^{(n+1)} |||_{n^{1/(\rho+\varepsilon)}} \le K_6 \exp\left( -\frac{2n\log n}{\rho+\varepsilon} + n + 2n \log(1+\varepsilon)  \right),
$$
Thus for large enough $n$ we obtain that
$$
|||u^{(n+1)} |||_{n^{1/(\rho+\varepsilon)}} \le \exp\left( -\frac{2n\log n}{\rho+2\varepsilon}\right).
$$
Now following the rest of the arguments after \cite[eq. (5.12)]{K_jost}, we get
\begin{equation*}
||A_n A_n^*-\bdone || + ||B_n|| \le \exp\left( -\frac{2n\log n}{\rho+2\varepsilon}\right) \quad \mbox{ for large } n.
\end{equation*}
This finishes the proof.
\end{proof}

\section{Geronimo--Baxter theorems for perturbations of periodic Jacobi operators}\label{sGB}

We want to prove the result~\eqref{GeronimoBaxter}, which relates the asymptotic behaviour of the Jacobi coefficients and the Fourier coefficients $c_n$ of $a(z)$ in~\eqref{meas}. This is done in Theorem~\ref{baxterSuper} below. This Jacobi--vs--Fourier analogy can be mimicked for finite range perturbations too: loosely speaking, Theorem~\ref{degree} says that having $n$ aperiodic Jacobi coefficients guarantees precisely $n$ extra singularities. The analogous result for finite Fourier/Taylor series (aka ``polynomials'', of course) is that having $n$ non-zero Fourier/Taylor coefficients guarantees precisely $n$ roots (fundamental theorem of algebra).




Apart from an interest of their own, our motivation for these Geronimo--Baxter type results stems from the inverse resonance problem in Section~\ref{sResonance}. Indeed, as we discuss later, the requirement of $a(z)$ to be of the  exponential growth of order less than $1$ is the right condition for the uniqueness of the inverse resonance problem to hold.

\subsection{Finite range perturbations}\label{ssGBFinite}


Let us introduce the notation
\begin{equation*}
\calJ^{(s)} = (a_{n+s},b_{n+s})_{n=1}^\infty,
\end{equation*}
that is, $\calJ^{(s)}$ is the Jacobi matrix obtained from $\calJ$ by removing the first $s$ rows and columns. Note that $\calJ^{(0)}$ is just $\calJ$.

\begin{definition} 
Let $\calT_\fre$ be the isospectral torus of Jacobi matrices associated with the finite gap set $\fre$.
\begin{itemize}
\item Denote by $\calT_\fre^{[2s-1]}$ the set of all matrices for which $\calJ^{(s)} \in \calT_\fre$, $\calJ^{(s-1)}\notin\calT_\fre$, and $a_s = a_{s+p}$;
\item Denote by $\calT_\fre^{[2s]}$ the set of all matrices for which $\calJ^{(s)} \in \calT_\fre$, but $\calJ^{(s-1)}\notin\calT_\fre$, and $a_s \ne a_{s+p}$.
\end{itemize}
\end{definition}
\begin{remarks}
1. One may want to think of the index $k$ in $\calT_\fre^{[k]}$ as the position of the last coefficient in the sequence $b_1,a_1,b_2,a_2,\ldots$ that fails to be periodic (it is important to put the $b$'s coefficients before the $a$'s here).

2. Thus the set of all eventually periodic matrices splits into the disjoint union $\calT_\fre^{[0]}\cup \calT_\fre^{[1]} \cup \ldots \cup \calT_\fre^{[k]} \cup\ldots$

3. $\calT_\fre^{[0]}$ coincides with $\calT_\fre$.
\end{remarks}

\begin{theorem}\label{degree}
Under one of the equivalent conditions of Theorem~\ref{sp_fin},
$$\calJ \in \calT_\fre^{[k]}$$
 if and only if
$$ \deg a(z) = k+p-1.$$
\end{theorem}
\begin{proof}
Let us start by observing that by Theorem~\ref{sp_fin}, $m$ is a map $\calS_\fre \to \bbC\cup\{\infty\}$. As such, the total number of its zeros (counted with multiplicities) equals to the total number of its poles (counted with multiplicities), which we will denote by $\deg m$. See~\cite[Cor~5.12.4]{Rice} for a proof.

The case $k=0$ of our theorem is well-known (originally due to Flaschka--McLaughlin--Krichever--van Moerbeke; see, e.g.,~\cite[Sect 5.13]{Rice} and the references therein). For a future reference we note that for any $\calJ\in \calT_\fre^{[0]}$, $m(z)$ has one singularity per gap, a first order zero at $\infty_+$, and a first order pole at $\infty_-$.

Now suppose $k=1$ or $k=2$, i.e., $\calJ^{(1)}\in \calT_\fre$.
The $m$-functions $m(z)$, $m^{(1)}(z)$ of $\calJ$ and $\calJ^{(1)}$ are known to obey
\begin{equation}\label{stripping}
m(z)=\frac{1}{b_1-z-a_1^2 m^{(1)}(z)}
\end{equation}
(this is essentially just the Schur complement formula). Since $\calJ^{(1)}\in\calT_\fre^{[0]}$, we know that
$m^{(1)}$ has one singularity per gap and a pole at $\infty_-$. Recall that $m^{(1)}(z)\sim -\tfrac{1}{z}$ at $\infty_+$, and  let $m^{(1)}(z)\sim k_{1}z+k_0+O(\tfrac1z)$, $k_1\ne0$, at $\infty_-$. Thus by~\eqref{stripping}, $m(z)=0$ exactly once per each gap, at $\infty_+$, and possibly at $\infty_-$. Note that $m(z)=0$ at $\infty_-$ if and only if $-1-a_1^2 k_1 \ne 0$. But we know that if $\calJ$ is periodic then $m(z)$ has a pole at $\infty_-$. Therefore $a_1=\sqrt{-1/k_1}$ is precisely the condition for $a_1$ to be periodic. In other words, $a_{p+1}=\sqrt{-1/k_1}$.

If $k=2$, then $a_1\ne a_{p+1}$, so $a_1 \ne \sqrt{-1/k_1}$, and therefore we just  showed that $m$ has exactly $p+1$ zeros: once per each gap, one at $\infty_+$, and one at $\infty_-$. Thus $\deg m = p+1$. Therefore there are precisely $p+1$ poles of $m$ \textit{all of which are finite}. In other words, $\deg a = p+1$.

If $k=1$, then $a_1=a_{p+1} =\sqrt{-1/k_1}$, but $b_1$ is aperiodic ($b_1\ne b_{p+1}$). Then $\infty_-$ is not a zero of $m$, so we have precisely $p$ zeros (one per each gap and $\infty_+$), i.e., $\deg m = p$. So $m$ has $p$ of poles. Are all of them finite? By \eqref{stripping},  $\infty_-$ is a pole of $m$ if and only if $b_1-a_1^2 k_0 =0$. Periodic Jacobi matrices have a pole at $\infty_-$, which means that $b_1+ k_0/k_1 =0$ is exactly the condition for $b_1$ to be periodic. Thus $b_{p+1}=-k_0/k_1$, and since $b_1\ne b_{p+1}$, we obtain that $\infty_-$ is not a pole of $m$, i.e., all of the $p$ poles of $m$ are finite. In other words, $\deg a = p$.

Now $k\ge 3$ follows easily by induction. First note that $\infty_-$ was not a pole in either of the cases $k=1$ or $k=2$ above. Therefore by~\eqref{stripping}, $\infty_-$ is always a zero when $k\ge3$. Using~\eqref{stripping} again we obtain that $m$ has zeros at $\infty_+$, $\infty_-$, and at every pole of $m^{(1)}$. Therefore $\deg m = \deg m^{(1)}+2$.
\end{proof}

%
%

\subsection{Super-exponential perturbations}\label{ssGBSuper}

\begin{theorem}\label{baxterSuper}
Suppose that one of the equivalent conditions of Theorem~\ref{sp} holds with $R=\infty$. The following are equivalent:
\begin{itemize}
\item[(1)] $a(z)$ is an entire function of growth order $\rho$;
\item[(2)] Taylor's coefficients $c_n$ of $a(z)$ satisfy
$$
\limsup_{n\to\infty} \frac{n\log n}{\log |c_n|^{-1}} = \rho
$$
\item[(3)] The Jacobi coefficients $(a_n,b_n)_{n=1}^\infty$ satisfy
$$
\limsup_{n\to\infty} \frac{n\log n}{\log (|a_n-a_n^{\circ}| + |{b}_n-b_n^{\circ}|)^{-1}} = \frac{\rho}{2p},
$$
where $(a_n^{\circ},b_n^{\circ})_{n=1}^\infty$ is a periodic Jacobi matrix from $\calT_\fre$.
\end{itemize}
\end{theorem}

\begin{remarks}
1. Compare this with Geronimo's~\cite[Thm 13]{Geronimo}. His result is for~\eqref{expFree} only and is restricted to the case when $a(z)$  has at most finitely many zeros. We stress however that his techniques of Beurling weighted Banach algebras allow a finer control over the asymptotics of the coefficients.

2. That (1) and (2) are equivalent is a very standard fact. We include (2) just for the aesthetical purpose, so that (3) looks more pleasing.
\end{remarks}
\begin{proof}
We briefly recall the notation from the author's~\cite{K_merom}. If $\Delta$ is the discriminant~\eqref{discriminant} associated with $\fre$, then $\Delta(\calJ)$ can be viewed as a block (see Section~\ref{sMatrix}) Jacobi operator. Let $\{A_n,B_n\}_{n=1}^\infty$ be its ($p\times p$ matrix-valued) Jacobi coefficients, $\mathfrak{m}(z)$ be its matrix-valued $m$-function, and $\mathfrak{a}(z)$ be the function from~\eqref{acPart}. Let $\mathfrak{p}_j$ ($j\ge 0$) be the $j$-th right matrix-valued orthonormal polynomial associated to $\Delta(\calJ)$. Let $p_j$ ($j\ge 0$) be the $j$-th (scalar) orthonormal polynomial associated to $\calJ$. Finally, define $f_1, \ldots, f_p$ to be the $p$ inverse functions of $\Delta(z)$ (i.e., $\Delta(z)=\lambda \Rightarrow z=f_j(\lambda)$) and  $\tilde{f}_1,\ldots,\tilde{f}_p$ to be their lifts as maps $\calS_{[-2,2]} \to \calS_\fre$ (see~\cite[Sect 4.1]{K_merom}).

A simple modification of the proof of~\cite[Lemma~4.3]{K_merom} (same arguments but without doing the summation in~\cite[Lemma~4.2]{K_merom}) produces
\begin{multline}\label{comput1}
[m(\tilde{f}_l(\lambda))-m^\sharp(\tilde{f}_l(\lambda))]
\left( \begin{array}{cccc} 1&p_1&\cdots & p_{p-1} \\ p_1&p^2_1&\cdots & p_1p_{p-1} \\ \vdots&\vdots&\ddots&\vdots \\ p_{p-1}&p_1p_{p-1} &\cdots & p^2_{p-1} \end{array}
\right) (f_l(\lambda))
\\
=  \left[\mathfrak{m}_\Delta(\lambda)-\mathfrak{m}^\sharp_\Delta(\lambda)\right] (K_{11,l}(\lambda)+\mathfrak{p}_1(\lambda) K_{21,l}(\lambda)),
\end{multline}
where  $K_{ij,l}(\lambda)$ is the $(i,j)$-th $p\times p$ block entry of $c_0 \prod_{j\ne l} (\calJ-f_j(\lambda))$. Now note that by Theorems~\ref{sp},~\ref{baxterMatrix}, and identity~\eqref{tildeF},
\begin{align}
\label{comput2} m(z)-m^\sharp(z) &= 2 \pi \frac{\sqrt{r(z)}}{a(z)},
\\
\label{comput3} \mathfrak{m}_\Delta(\lambda)-\mathfrak{m}^\sharp_\Delta(\lambda) &= 2 \pi \mathfrak{a}(\lambda)^{-1} \sqrt{\lambda^2-4}.
\end{align}

Combining~\eqref{comput1},~\eqref{comput2},~\eqref{comput3}, together with $\Delta(z)\sim z^p, z\to\infty$, one obtains
\begin{equation*}
\begin{aligned}
\limsup_{r\to\infty} \frac{\log\log \sup_{|z|=r} |a(z)|}{\log r}  &= \limsup_{r\to\infty} \frac{\log\log \sup_{|z|=r} ||\mathfrak{a}(\Delta(z))||}{\log r} \\
& = p \limsup_{r\to\infty} \frac{\log\log \sup_{|z|=r} ||\mathfrak{a}(z)||}{\log r}.
\end{aligned}
\end{equation*}

By the same arguments as \cite[Lemma~B.3]{K_merom} (which in turn were adopted from \cite{DKS}), one can show that
\begin{equation*}\label{Lip}
\limsup_{n\to\infty} \frac{n\log n}{\log (|a_n-a_n^{\circ}| + |{b}_n-b_n^{\circ}|)^{-1}} = \limsup_{n\to\infty} \frac{n\log n}{\log (||A_n A_n^*-\bdone || + ||B_n|| )^{-1}}
\end{equation*}
(in the course of showing this, one needs to use the fact that for any $\eta$ there exists $M$ such that $1\le e^{\eta n \log n} \sum_{k=n}^\infty e^{-\eta k \log k} \le M$).

Now combining the last two equations together with Theorem~\ref{baxterMatrix}, we obtain that (3) holds if and only if
$$
\limsup_{r\to\infty} \frac{\log\log \sup_{|z|=r} |a(z)|}{\log r}=\rho,
$$
which is the definition of (1). That this is equivalent to (2) is a well-known fact.
\end{proof}

\section{Inverse resonance problem: existence and uniqueness}\label{sResonance}

With the above results at hand, it is now easy to solve the inverse resonance problem. The idea is simple: resonances and eigenvalues recover the function $a(z)$, which determines the spectral measure.

\subsection{Inverse resonance problem for finite range perturbations: existence and uniqueness}\label{ssResonanceFinite}
\begin{theorem}\label{resonanceEv}
Let $\{R_j\}_{j=1}^{K}$ and $\{E_j\}_{j=1}^N$ $(0\le N,K< \infty)$ be two sequences of complex numbers $($possibly with multiplicities$)$. These two sequences are respectively resonances and eigenvalues of an eventually periodic Jacobi matrix if and only if 
\begin{itemize}
\item[$(O_1)$] $\{E_j\}_{j=1}^N$ oddly interlace with $\{R_j\}_{j=1}^{K}$ on $\bbC$ $($see Def.~\ref{OI_light}$)$;\footnote{We remind that $(O_1)$ includes $\{R_j\}_{j=1}^K\cap\{E_j\}_{j=1}^N = \varnothing$ as part of the Definition~\ref{OI_light}.}
\item[$(O_2)$] Each gap $[\beta_k,\alpha_{k+1}]$ contains an odd number of points from $\{E_j\}_{j=1}^N \cup \{R_j\}_{j=1}^{K}$ $($counting with multiplicities$)$;
\item[$(O_3)$] $E_j\in\bbR\setminus\fre$ for every $j$; each $E_j$ is of multiplicity $1$;
\item[$(O_4)$] $R_j\in\bbC\setminus\Int(\fre)$ and they are real or come in complex conjugate pairs $($counting multiplicities$)$; if $R_{j}\in\cup_{j=1}^p\{\alpha_j,\beta_j\}$, then the multiplicity of $R_{j}$ is $1$.
\end{itemize}
 Such a Jacobi matrix $\calJ$ is unique.

In fact, $\calJ$ is in $\calT_\fre^{[K+N-p+1]}$.
\end{theorem}

\begin{remarks}
1. There is an implicit condition here that $K+N \ge p-1$, since by $(O_2)$ there must be at least one singularity per gap.


2. $\calT_\fre=\calT_\fre^{[0]}$ corresponds to the total number of singularities being $p-1$. In view of $(O_2)$, this means that there must be one singularity per gap, each of which could be an eigenvalue or a resonance. Since $\pi^{-1}([\beta_j,\alpha_{j+1}])$ is homeomorphic to a circle, we obtain that $\calT_\fre^{[0]}$ is homeomorphic to the direct product of $p-1$ circles, i.e., a $(p-1)$-torus. This justifies the term ``torus'' in Definition \ref{torus}.

3. We stress that $R_j$ and $E_j$ are finite numbers. In particular, we do not call $\infty_-$ a ``singularity'' here even if it happens to be a pole of $m$ (see Definition~\ref{singularity}). In fact, it is clear from the proof of Theorem~\ref{degree} that $\infty_-$ is indeed a pole of $m$ if $K+N = p-1$, but it is a regular point otherwise.
\end{remarks}
\begin{proof}
Let us first show the necessity. We already showed in Theorem~\ref{sp_fin} that $(O_1)$ and $(O_3)$ holds for any eventually periodic Jacobi matrix (see $(S'_b)$). $(O_2)$ follows from the sign-alternating property of $a(z)$, see $(S'_a)$. That $R_j$ are real or come in complex conjugate pairs follows from real-analyticity of $a(z)$. The rest of $(O_4)$ is a consequence of integrability of $\tfrac{d\mu}{dx}$ on $\fre$.

To show sufficiency, given $\{R_j\}_{j=1}^{K}$ and $\{E_j\}_{j=1}^N$, let
$$a(z)=A \prod_{j=1}^{K} (z-R_j) \prod_{j=1}^{N} (z-E_j),$$ where $A$ is a real constant to be determined momentarily. First choose the sign of $A$ so that $a(z)$ is positive on $(\alpha_p,\beta_p)$. Using $(O_2)$, we can see that $a(z)$ satisfies $(S_a)$ of Theorem~\ref{sp_fin}. Define $w_j>0$ by \eqref{weights2} for each $1\le j\le N$. 
Finally, the absolute value of $A$ can be chosen so that the total mass of $\mu$ is $1$.

Uniqueness follows from the fact that each step of the measure reconstruction was uniquely determined by the spectral characterization of Theorem~\ref{sp_fin}. The fact that $\calJ\in\calT_\fre^{[K+N-p+1]}$ follows from Theorem~\ref{degree}.
%
\end{proof}
\begin{corollary}
Let
$$
S = \fre \cup \{E_1,\ldots E_L\}, \quad L<\infty, \quad E_j\in\bbR\setminus \fre.
$$
Then there exists an eventually periodic Jacobi matrix $\calJ$ with
$$
\sigma(\calJ) = S.
$$
\end{corollary}
\begin{remark}
Of course such $\calJ$ is far from being unique because we can choose resonances in infinitely many ways.
\end{remark}

\subsection{Inverse resonance problem for super-exponentially decaying perturbations: existence}\label{ssResonanceExistence}

One can now fully analyze the inverse resonance problem in the case of exponential perturbations as well. We will restrict ourselves to the case of super-exponential perturbations though, since it is cleaner and satisfies uniqueness, which makes it more interesting, more natural, and more satisfying. We just note that for the class of exponential perturbations~\eqref{exp} (with $R<\infty$), the uniqueness of the inverse problem clearly does not hold, and the necessary and sufficient conditions for the existence are the same as in the next theorem but with: $\calE$ instead of $\bbC$ in $(O_1)$; ``If $[\beta_k,\alpha_{k+1}]$ belongs to $\calE$'' added in the beginning of $(O_2)$; $(O_5)$ modified to state that the set of resonances contains no accumulation point in $\operatorname{Int}(\calE)$. We leave the details to the reader and for the rest of the section restrict ourselves to the class of super-exponential perturbations, that is, all the Jacobi matrices satisfying~\eqref{exp} with $R=\infty$.



\begin{theorem}\label{resonanceExp}
Let $\{R_j\}_{j=1}^{K}$ $(0\le K\le \infty)$ and $\{E_j\}_{j=1}^N$ $(0\le N< \infty)$ be two sequences of complex numbers  (possibly with multiplicities). These two sequences are respectively resonances and eigenvalues of a Jacobi operator from~\eqref{exp} with $R=\infty$ if and only if
they satisfy 
\begin{itemize}
\item[$(O_1)$] $\{E_j\}_{j=1}^N$ oddly interlace with $\{R_j\}_{j=1}^{K}$ on $\bbC$ $($see Def.~\ref{OI_light}$)$;\footnote{Again, we remind that $(O_1)$ includes $\{R_j\}_{j=1}^K\cap\{E_j\}_{j=1}^N = \varnothing$ as part of the Definition~\ref{OI_light}.}
\item[$(O_2)$] Each gap $[\beta_k,\alpha_{k+1}]$ contains an odd number of points from $\{E_j\}_{j=1}^N \cup \{R_j\}_{j=1}^{K}$ $($counting with multiplicities$)$;
\item[$(O_3)$] $E_j\in\bbR\setminus\fre$ for every $j$; each $E_j$ is of multiplicity $1$;
\item[$(O_4)$] $R_j\in\bbC\setminus\Int(\fre)$ and they are real or come in complex conjugate pairs $($counting with multiplicities$)$; if $R_{j}\in\cup_{j=1}^p\{\alpha_j,\beta_j\}$, then the multiplicity of $R_{j}$ is $1$;
\item[$(O_5)$] If $K=\infty$ then $\lim_{j\to\infty} |R_j| = \infty$.
\end{itemize}
\end{theorem}
\begin{proof}
The necessity follows from Theorem~\ref{sp} just as in Theorem~\ref{resonanceEv}. $(O_5)$ follows from the fact that $\{E_j\}_{j=1}^N \cup \{R_j\}_{j=1}^{K}$ are the zeros of an entire function.

To show sufficiency, given $\{R_j\}_{j=1}^{\infty}$ and $\{E_j\}_{j=1}^N$, form an entire function $a(z)$ having $\{R_j\}_{j=1}^{K}\cup\{E_j\}_{j=1}^N$ as the set of its zeros (repeated according to their multiplicities). Indeed, due to the condition $\lim_{j\to\infty} |R_j| = \infty$, this can be done by forming a convergent infinite product of primary factors (theorem of Weierstrass, see, e.g.,~\cite[Sect~1.3]{bLevin}). The rest of the argument is analogous to the proof of Theorem~\ref{resonanceEv}.

Note that we can form many Weierstrass products with the same locations of zeros, which means that the uniqueness of the inverse resonance problem will not hold unless we restrict the class of Jacobi operators. We do this in the next subsection.
\end{proof}

\subsection{Inverse resonance problem for super-exponentially decaying perturbations: uniqueness}\label{ssResonanceUniqueness}
As we saw above, we reduced the inverse resonance problem to the problem of recovering an entire function from the locations of its zeros. This, of course, is classical and well-known. Combining this with our Baxter-type theorem, we obtain the following statement.

\begin{theorem}
Consider the class $\mathfrak{R}$ of Jacobi matrices $(a_n,b_n)_{n=1}^\infty$ satisfying
\begin{equation}\label{classR}
\limsup_{n\to\infty} \frac{n\log n}{\log (|a_n-a_n^{\circ}| + |{b}_n-b_n^{\circ}|)^{-1}} < \frac{1}{2p},
\end{equation}
where $(a_n^{\circ},b_n^{\circ})_{n=1}^\infty$ is a periodic Jacobi matrix from $\calT_\fre$.

Let $\{R_j\}_{j=1}^{K}$ $(0\le K\le \infty)$  and $\{E_j\}_{j=1}^N$ $(0\le N< \infty)$ be two sequences of complex numbers $($possibly with multiplicities$)$.  These two sequences are respectively resonances and eigenvalues of a Jacobi operator from the class $\mathfrak{R}$ if and only if 
they satisfy $(O_1)$--$(O_5)$ of Theorem~\ref{resonanceExp} and
\begin{equation}\label{sumResonances}
\sum_{j=1}^K \frac{1}{|R_j|^\alpha} < \infty
\end{equation}
for some $0<\alpha<1$. Moreover, such a Jacobi matrix is unique.
\end{theorem}
\begin{remarks}
1. The set $\mathfrak{R}$ is very natural here: e.g., allowing perturbations with $\le$ in~\eqref{classR} would mean allowing all the entire functions $a(z)$ of order $1$, which would allow two entire functions to have the same sets of zeros. This would violate the uniqueness.

2. To restate~\eqref{classR}, $\mathfrak{R}$ consists of perturbations for which there exists $\eta>2p$ and $C>0$ such that
\begin{equation}\label{classR2}
|a_n-a_n^{\circ}| + |{b}_n-b_n^{\circ}| \le C e^{-\eta n \log n}.
\end{equation}
Typically authors restrict themselves to a narrower class of perturbations $|a_n-a_n^{\circ}| + |{b}_n-b_n^{\circ}| \le C e^{-n^{\beta}}, \beta>1,$ in order to get uniqueness.
\end{remarks}
\begin{proof}
For the operators of class $\mathfrak{R}$, the function $a(z)$ is of order strictly less than 1 (Theorem~\ref{baxterSuper}). By Hadamard's factorization theorem (see, e.g.,~\cite[Sect~2.7]{bBoas}), any such a function is uniquely determined by the set of its zeros. Moreover, zeros of these functions are precisely characterized by the condition~\eqref{sumResonances}, see, e.g.,~\cite[Sect~2.9]{bBoas}.
\end{proof}

\section{Inverse resonance problem: stability}\label{sResonanceStability}

We want to obtain a result that states that if resonances and eigenvalues of two Jacobi operators are pairwise close to each other, then their Jacobi coefficients are also close. In order for this to have any chance of success, we must restrict ourselves to the class of the Jacobi operators which are uniquely recoverable from the set of its resonances and eigenvalues. From the discussion in the previous section, we are led to consider the class of matrices $\frakR$, see~\eqref{classR2}. Let us now 
impose uniform bounds on the coefficients from above and from below.

\begin{definition}
Given $\eta>2p$, $C>0$, and $\gamma>0$, define $\frakR_{\eta,C,\gamma}$ to be the class of Jacobi matrices satisfying the following two conditions:
\begin{equation}\tag{H1}\label{H1}
|a_n-a_n^{\circ}| + |{b}_n-b_n^{\circ}| \le C e^{-\eta n \log n}
\end{equation}
for some  periodic Jacobi matrix $(a_n^{\circ},b_n^{\circ})_{n=1}^\infty$ from $\calT_\fre$ and
\begin{equation}\tag{H2}\label{H2}
a_n\ge \gamma.
\end{equation}
\end{definition}
\begin{remarks}


1. The conditions ~\eqref{H1} and ~\eqref{H2} are in fact necessary in order to obtain the stability result. Indeed, as we are about to see, if either ~\eqref{H1} or ~\eqref{H2} is not imposed, then we could find a Jacobi matrix with many resonances arbitrary close to $\pm2$. This would violate any hope for the stability as the following simple counterexample explains. Take $\fre=[-2,2]$ and take the spectral measure $\mu_n$~\eqref{meas} with 4 resonances  at $-2,-2-\frac1n, 2,2+\frac{1}{n^2}$ only, and $\tilde{\mu}_n$ with 4 resonances  at $-2, -2-\frac{1}{n^2}, 2,2+\frac{1}{n}$ only. As $n$ becomes large, these measures tend to the delta functions at $2$ and $-2$, respectively. Thus the first Jacobi coefficients $b^{(n)}_1$ and $\tilde{b}^{(n)}_1$ of $\mu_n$ and $\tilde{\mu}_n$ are never close to each other $b^{(n)}_1-\tilde{b}^{(n)}_1 = \int x d\mu_n(x) - \int x d\tilde{\mu}_n(x) \to 4$, even though the distances between the resonances do become infinitely small $| \tfrac1n - \tfrac{1}{n^2} | \to 0$.
\end{remarks}

We are interested in $\frakR_{\eta,C,\gamma}$ to be as large as possible, so without loss of generality we assume that the constant $C$ is large enough, and $\gamma$ and $\eta$ are close to $0$ and $2p$, respectively. Let us define a constant
\begin{equation*}
\tau = 1-\frac{2p}{\eta}
\end{equation*}
that will naturally appear in the estimates below, $0<\tau<1$.

\begin{theorem}\label{thmStability}
Consider $\frakR_{\eta,C,\gamma}$ for some $C>0,\eta>2p,\gamma>0$. There exist $Q>1$, $\varepsilon_0>0$, and $R_0>0$ $($depending on $\fre,\eta,C,\gamma$ only$)$ such that for any $0<\varepsilon<\varepsilon_0$ and $R>R_0$ the following holds.

Choose any $\calJ=(a_n,b_n)_{n=1}^\infty$ and $\tilde\calJ=(\tilde{a}_n,\tilde{b}_n)_{n=1}^\infty$ in $\frakR_{\eta,C,\gamma}$. Let $\{E_j\}_{j=1}^{N_e}$ and $\{\til{E}_j\}_{j=1}^{N_e}$ $(N_e<\infty)$ be the eigenvalues of $\calJ$ and $\til\calJ$, respectively. Let $\{R_j\}_{j=1}^{N_R}$ and $\{\til{R}_j\}_{j=1}^{N_R}$ $($repeated according to their multiplicities$)$ $(N_R<\infty)$ be those resonances of $\calJ$ that lie in the disk $ R \bbD = \{z: |z|\le R\}$. If
\begin{align*}
|E_j-\til{E}_j| < \varepsilon & \quad \mbox{ for all } j, \\
|R_j-\til{R}_j| < \varepsilon & \quad \mbox{ for all } j,
\end{align*}
then
\begin{equation}\label{estim}
|a_n - \til{a}_n | + |b_n - \til{b}_n| \le Q^{n^2} \left(\frac{1}{R^\tau} + \sqrt{\veps}\right) \quad \mbox{ for all } n.
\end{equation}
\end{theorem}
\begin{remarks}
1. As is clear from the proof, $\sqrt{\veps}$ is sharp here. It is caused by the resonances close to the endpoints of $\fre$. If one forbids resonances in small neighbourhoods of the endpoints, then the right-hand side of~\eqref{estim} can be improved to $Q^{n^2}\left(\frac{1}{R^\tau} + \veps\right)$.

2. Resonances of $\calJ$ and $\tilde\calJ$ outside of $R\bbD$ do not have to be close to each other.

3. One can also allow an eigenvalue $E_j$ of $\calJ$ migrating into a resonance $\til{R}_j$ of $\til\calJ$, but the distance from $E_j$ to $\til{R}_j$ should then be measured not as $|E_j-\til{R}_j|<\veps$, but rather as $\operatorname{dist}(E_j,\fre)+\operatorname{dist}(\til{R}_j,\fre)<\veps$ (which is natural if one thinks about the surface $\calS_\fre$). To accommodate such a situation, just apply our theorem twice: first move the eigenvalue $E_j$ to the closest endpoint of $\fre$. It ceases being an eigenvalue since the weight~\eqref{weights2} becomes zero, but all the estimates in the proof still work. Then apply the theorem again to move the resonance from the endpoint to $\til{R}_j$.

4. If one narrows $\frakR_{\eta,C,\gamma}$ by replacing~\eqref{H1} with $|a_n-a_n^{\circ}| + |{b}_n-b_n^{\circ}| \le C e^{-n^{\beta}}$ with some $\beta>1$, then~\eqref{estim} can be improved to $$
|a_n - \til{a}_n | + |b_n - \til{b}_n| \le Q^{n^2}\left(\frac{(\log R)^{\beta/(\beta-1)}}{R} + \veps\right)
$$
by following the exact same proof.
\end{remarks}

We prove the theorem in the end of the section after establishing a series of lemmas. We start with a collection of some elementary inequalities.

\begin{lemma}\label{lemTech}
\begin{itemize}
\item[(a)] For any $q\in\bbC$,
$$
-\big| \tfrac{1-q}{q} \big| \le \log |q| \le |1-q|.
$$
\item[(b)] For any $q_j\in\bbC$,
$$\Big|1-\prod_{j} |q_j| \Big| \le \Big(\sum_j |1-q_j|\Big)e^{\sum_j |1-q_j|}.$$
\item[(c)] If $|q_j|\le H$, $ |\tilde{q}_j|\le H$, and $|q_j-\tilde{q}_j|\le e$, then
$$
|q_1\ldots q_n - \tilde{q}_1\ldots \tilde{q}_n|< n H^{n-1} e.
$$
\item[(d)] If $1/L\le |q_j|\le H$, $1/H\le |\tilde{q}_j|\le H$, and $|q_j-\tilde{q}_j|\le e$, then
$$
\left| \sqrt{\tfrac{q_1 q_2}{q_3 q_4}} - \sqrt{\tfrac{\tilde{q}_1 \tilde{q}_2}{\tilde{q}_3 \tilde{q}_4}} \right| \le 2 H^4 e.
$$
\end{itemize}
\end{lemma}
\begin{proof}
The right-hand side inequality in (a) is standard. The left-hand side inequality is obtained from the right-hand side by plugging in $q^{-1}$ instead of $q$.

(b) can be obtained by putting $q=\prod q_j$ into the elementary inequality $|1-|q|| \le |\log|q|| e^{\log|q|}$ and then using the right-hand side inequality in (a).

(c) can be obtained by adding and subtracting $q_1\ldots q_k \tilde{q}_{k+1}\ldots \tilde{q}_n$ for every $k$.

Finally,
$$
\left| \sqrt{\tfrac{q_1 q_2}{q_3 q_4}} - \sqrt{\tfrac{\tilde{q}_1 \tilde{q}_2}{\tilde{q}_3 \tilde{q}_4}} \right| = \frac{|\sqrt{q_1 q_2 \tilde{q}_3 \tilde{q}_4} - \sqrt{\tilde{q}_1 \tilde{q}_2 q_3 q_4} | }{\sqrt{q_3 q_4 \tilde{q}_3 \tilde{q}_4}} \le H^2 |\sqrt{q_1 q_2 \tilde{q}_3 \tilde{q}_4} - \sqrt{\tilde{q}_1 \tilde{q}_2 q_3 q_4} |.
$$
Now note that $|\sqrt{q}_j - \sqrt{\tilde{q}_j}| \le \tfrac12 \sqrt{H} e$ and apply (c) with $\sqrt{q_j}$'s instead of $q_j$'s to finish the proof of (d).
\end{proof}

Throughout this section there will be various positive constants appearing which we will start denoting by $M_j$ ($j=0,1,2,\ldots$). We will have to be careful to make sure that each of these constants is uniform, that is, they will all depend on $\fre,\eta,C, \gamma$ (and possibly on the choice of the preceding $M_j$'s), but do not depend on a specific choice of $\calJ$ from $\RR$.

\begin{lemma}\label{lemCoef}
For any $\eta>2p,C>0,\gamma>0$ there exists a constant $M_0> 0$ such that for any matrix $\calJ=(a_n,b_n)_{n=1}^\infty$ from $\frakR_{\eta,C,\gamma}$, the following holds true:
\begin{itemize}
\item[(a)] $|a_n|+|b_n| \le M_0$ for all $n$;
\item[(b)] For $(a_n^{\circ},b_n^{\circ})_{n=1}^\infty$ from ~\eqref{H1},
$$
-M_0 \le \sum_{j=1}^{\infty} \log\Big(\frac{a_j}{a_j^{\circ}}\Big) \le M_0,
$$
where the series converges absolutely.
\end{itemize}
\end{lemma}
\begin{proof}
First of all note that for any $(a_n^{\circ},b_n^{\circ})_{n=1}^\infty = \calJ^{\circ}\in\calT_\fre$,
\begin{equation}\label{sEq1}
|a_j^{\circ}|, |b_j^{\circ}| \le ||\calJ^{\circ}|| = \sup \{|x|: x\in\fre\}.
\end{equation}
This gives a uniform upper bound on $|a_j^{\circ}|$ and $|b_j^{\circ}|$. Combining this with
\begin{equation*}\label{capacity}
\prod_{j=1}^{p} a_j^{\circ} = \operatorname{Cap}(\fre)^p,
\end{equation*}
where $\operatorname{Cap}(\fre)$ is the logarithmic capacity of the set $\fre$, see~\cite[Thm 5.5.17]{Rice},
we also obtain a uniform lower bound on $a_j^{\circ}$ for $\calJ^{\circ}\in\calT_\fre$.

Using \eqref{sEq1} and ~\eqref{H1}, we obtain (a).

\smallskip

Now use Lemma~\ref{lemTech}(a) to see that
\begin{equation}\label{sEq2}
-\sum_{j=1}^{\infty} \frac{|a_j-a_j^{\circ}|}{a_j} \le \sum_{j=1}^{\infty} \log\Big(\frac{a_j}{a_j^{\circ}}\Big) \le \sum_{j=1}^{\infty} \frac{|a_j-a_j^{\circ}|}{a_j^{\circ}}.
\end{equation}
Combining this with the lower bounds on $a_j$ and $a_j^{\circ}$, and using ~\eqref{H1}, we obtain a uniform bound in (b). The absolute convergence of the series in (b) can be shown to be equivalent to the convergence of the right-hand side of~\eqref{sEq2}.
\end{proof}

As before, let us view $\Delta(\calJ)$ as a block Jacobi matrix with the $p\times p$ matrix-valued Jacobi coefficients $(A_n,B_n)_{n=1}^\infty$, and let $\mathfrak{a}(z)$, see~\eqref{acPart}, be the denominator of the a.c. part of the spectral measure of $\Delta(\calJ)$. Finally, let $|A|=\sqrt{A^* A}$.

For any analytic function $g(z)$ we denote the zero counting function by
\begin{equation*}
\nu_g(R) := \sharp\{z: g(z)=0 \mbox{ and } |z|\le R\},
\end{equation*}
where we count the zeros according to their multiplicities.

\begin{lemma}\label{lemMV}
For any $\eta>2p,C>0,\gamma>0$ there exist a constant $M_8>0$ such that for any matrix $\calJ=(a_n,b_n)_{n=1}^\infty$ from $\frakR_{\eta,C,\gamma}$, the following holds true:
\begin{itemize}
\item[(a)] 
$||A_n|| +||B_n|| \le M_8$ for all $n$;
\item[(b)] $||A_n A_n^*-\bdone || + ||B_n|| \le M_8 e^{-\eta n\log n}$ for all $n$;
\item[(c)] For any $n$,
\begin{equation}\label{partC}
\Big| \sum_{j=1}^n \log\det|A_j| - p \sum_{j=1}^{pn} \log\Big(\frac{a_j}{a_j^{\circ}}\Big) \Big| \le M_8;
\end{equation}
\item[(d)]  $ ||A_n^{-1}|| \le M_8$; $\prod_{j=1}^n ||A_j^{-1}|| \le M_8$ for all $n$;
\item[(e)] $\nu_{\det\mathfrak{a}}(r)\le M_8 r^{2/\eta}$ for all $r>1$.
\end{itemize}
\end{lemma}
\begin{proof}
$||\Delta(\calJ)||$ can be easily uniformly bounded from above by using a bound on $||\calJ||$, see Lemma~\ref{lemCoef}(a). This proves part (a).

\smallskip

For (b) we need to reuse the arguments from~\cite[Thm~11.13(vi)$\Rightarrow$(i)]{DKS} or~\cite[Lemma~B.3]{K_merom}: each of the entries of $A_n-\bdone$ and $B_n$ is a polynomial function of $p$ consecutive pairs of $(a_j,b_j)$, and each of these functions (except for $B_1$) vanish when $\calJ\in\calT_\fre$ by the ``Magic Formula'' of Damanik--Killip--Simon~\cite{DKS}. Since each variable lies in a compact set ($a_j\in [\gamma,M_0], b_j\in[-M_0,M_0]$ by the previous lemma), we can uniformly bound each of the partial derivatives of these functions. Therefore the Lipschitz property gives us a uniform bound $||A_n-\bdone|| + ||B_n|| \le M_1 e^{-\eta n\log n}$. This means that starting from some $n\ge M_2$ (uniformly!), $||A_n-\bdone||\le 1/2$. This allows us to use~\cite[Prop 11.12]{DKS} which produces $||A_n A_n^* - \bdone || +||B_n|| \le M_3 (||A_n-\bdone||+||B_n||) \le M_1 M_3 e^{-\eta n\log n}$ for $n \ge M_2$. 
Values $||A_n A_n^* - \bdone || + ||B_n||$ for $1\le n < M_2$ can be easily incorporated into the estimate by using (a).

\smallskip

For (c), we use~\cite[Eq. (11.47)]{DKS} which states that
\begin{equation}\label{det|An|}
\det|A_n| = \prod_{j=(n-1)p+1}^{np} \prod_{k=j}^{j+p-1}  \Big(\frac{a_j}{a_j^{\circ}}\Big)
\end{equation}
to see that the difference in~\eqref{partC} is bounded above by $p(p-1) \sup_{j} \log|a_j|$, which can be uniformly bounded since $\gamma\le a_j\le M_0$.

\smallskip

Note that 
$||A_j^{-1}||$ is equal to the inverse of the minimal eigenvalue of $|A_j|$ which is $\le \frac{||A_j||^{p-1}}{\det|A_j|}$. Now, $\det|A_n|$  can be bounded below using~\eqref{det|An|}, and $\prod_{j=1}^n \det|A_j|$ can be bounded below by combining part (c) with Lemma~\ref{lemCoef}(b). An upper bound for $||A_n||$ was already established in (a). Finally, use Lemma~\ref{lemTech}(a) to get
$$
\prod_{j=1}^n ||A_j|| \le e^{\tfrac12 \sum_{j=1}^n \big| 1- ||A_j A_j^*|| \big|} \le e^{\tfrac12 \sum_{j=1}^n || \bdone - A_j A_j^*|| }.
$$
Using (b) and combining this all together proves (d).

\smallskip

Now let us reuse that arguments that lead us from~\eqref{superexp} to~\eqref{normU} (but now with $\eta$ instead of $\tfrac{2}{\rho+\veps}$). We need to justify why the constants $K_1$ and $K_2$ in~\eqref{refer1} and~\eqref{refer2} can be chosen uniformly. Comparing this with~\cite[Eq. (4.15)]{K_jost} and~\cite[Eq. (4.15)]{K_jost}, we see that this amounts to uniform upper bounds on $\prod_{j=1}^n ||A_j^{-1}||$ and $\sup_j ||A_j^{-1}||$, which is exactly what part (d) was for. Therefore~\eqref{normU} holds with a uniform constant $K_3$. Let us split the sum in~\eqref{normU} into $\sum_{j=1}^{n_0} + \sum_{j>n_0}$ with $n_0:=\lfloor 2^{1/\eta} |z|^{2/\eta} \rfloor$. If $n>n_0$ then $|z|^{2n} e^{-\eta n\log n} \le \big(\tfrac12\big)^n$. If $1\le n \le n_0$ then one can check that $|z|^{2n} e^{-\eta n\log n}$ is maximal when $n=\tfrac1e |z|^{2/\eta}$. Using these estimates we arrive at
\begin{multline*}
||u(z)|| \le K_3 |z|^2 (n_0 \max_{1\le n\le n_0}\{|z|^{2n} e^{-\eta n \log n} \} + 1) \\
 \le K_3 |z|^2 ( 2^{\frac1\eta} |z|^{\frac2\eta} e^{\frac{\eta}{e} |z|^{\frac2\eta}} +1 )
\le M_4 e^{\frac\eta2 |z|^{\frac2\eta}}
\end{multline*}
for a sufficiently large $M_4$.

This implies that $|\det u(z) | = \det |u(z) | \le ||u(z)||^p \le M^p_4 e^{p\frac{\eta}2 |z|^{\frac\eta2}}$. Also note that $|\det u(0) | = \prod_{j=1}^\infty |\det A^{-1}_j|$ (this follows, e.g., from the recurrence~\cite[Eq. (4.7)]{K_jost}). Therefore $-M_5 \le \log|\det u(0)| \le M_5$ from part (c) and Lemma~\ref{lemCoef}(b). Now we can use Jensen's theorem (see, e.g.,~\cite[Thm I.5.5]{bLevin} and~\cite[Lemma I.5.4]{bLevin}) to obtain
\begin{multline*}
\nu_{\det u}(r) \le \int_r^{er} \frac{\nu_{\det u}(t)}{t}dt \le \frac{1}{2\pi} \int_0^{2\pi} \log |\det u(er e^{i\theta})| d\theta - \log|\det u(0)| \\
\le p\log M_4 + \tfrac{p\eta}{2} (er)^{\frac2\eta} + M_5
\end{multline*}
for any $r$. Therefore for some uniform constant $M_6>0$, $\nu_{\det u}(r) \le M_6 r^{2/\eta}$ for all $r\ge 1$.

Now recall~\eqref{funcA} which implies $\det \mathfrak{a}(z+z^{-1}) = (2\pi)^p \det u(z) \det u(z^{-1})$. 
Observe that because of the trivial identity $(z-z_0)(z^{-1}-z_0) = -z_0(z+z^{-1}-z_0-z_0^{-1})$, for any $r>1$ there is a one-to-one correspondence (counting with multiplicities) between the zeros of $\det u(z)$ in the annulus $\{z\in\bbC: r^{-1} < |z| < r\}$ and the zeros of $\det \mathfrak{a}(z)$ in the ellipse $\{z+z^{-1}: |z|<r\}$. Since $r\bbD \subset \{z+z^{-1}: |z|<2r\}$ for large enough $r$, we get
$$
\nu_{\det \mathfrak{a}} (r) \le \sharp\{z: (2r)^{-1} <|z|< 2r, \det u(z) = 0 \} \le \nu_{\det u}(2r) \le M_7 r^{2/\eta}.
$$
%
\end{proof}

According to Theorem~\ref{sp}, for a $\calJ$ in $\frakR_{\eta,C,\gamma}$ its spectral measure is of the form~\eqref{meas} with $R=\infty$. Let $\{R_j\}_{j=1}^K$ ($K\le \infty$) be the resonances of $\calJ$, and $\{E_j\}_{j=1}^N$ ($N<\infty$) be its eigenvalues (eigenvalues oddly interlace with resonances on $\bbC$).

For the rest of the section let us fix the notation
\begin{align*}
U_{\delta_0}(z_0) &:=\{z: |z-z_0|<\delta_0\}, \\
W_{\delta_0}(\fre) &:=\{z: \real z \in \fre, -\delta_0\le \imag z\le \delta_0 \}, \\
U_{\delta_0}(\fre) &:=  \cup_{j=1}^p U_{\delta_0}(\alpha_j) \cup_{j=1}^p U_{\delta_0}(\beta_j) \cup W_{\delta_0}(\fre).
\end{align*}

\begin{lemma}\label{lemSp}
For any $\eta>2p,C>0,\gamma>0$ there exist constants $r_0>0$, $\delta_0>0$, $M_{15}>0$ such that for any matrix $\calJ=(a_n,b_n)_{n=1}^\infty$ from $\frakR_{\eta,C,\gamma}$, the following holds true:
\begin{itemize}
\item[(a)] $\sigma(\calJ)\subset \tfrac{r_0}{2}\bbD$, in particular $|E_j| < r_0/2$ for all $j$;
\item[(b)] The total number of eigenvalues $N\le M_{15}$;
\item[(c)] For any $r \ge r_0$,
\begin{equation*}\label{nu}
\nu_a (r) \le M_{15} r^{1-\tau};
\end{equation*}
\item[(d)] For any $r \ge r_0$,
\begin{equation*}
\sum_{j: |R_j|>r} \frac{1}{|R_j|} \le M_{15} r^{-\tau};
\end{equation*}
\item[(e)] There is at most one singularity $($i.e., zero of $a(z)$$)$ in each $U_{\delta_0}(\alpha_j)$ and at most one singularity in each $U_{\delta_0}(\beta_j)$, $j=1,\ldots,p$;
\item[(f)] There are no singularities in $W_{\delta_0}(\fre)$;
\item[(g)] $|E_j - E_k| \ge \delta_0$ for every $j\ne k$; $|E_j-R_k|\ge \delta_0$ for every $j,k$.
\end{itemize}
\end{lemma}
\begin{remarks}
1. Parts (c) and (d) are well-known for functions $a$ of exponential order smaller than $1$. We have to go carefully through all these lengths to make sure that the constant $M_{15}$ is uniform which is a non-trivial fact.

2. By shrinking $\delta_0$ if necessary we may assume that $\delta_0<\min\{\frac{\alpha_{j+1}-\beta_j}{2},\frac{\beta_j-\alpha_j}{2}\}$.
\end{remarks}
\begin{proof}
Part (a) follows from Lemma~\ref{lemCoef}(a).

\smallskip

Part (b) is a special case of part (c) since $N \le \nu_a(r_0/2) \le \nu_a(r_0)$ by part (a). 

\smallskip

From~\cite[Lemma 4.7]{K_merom} or~\cite[Prop 11.3]{DKS} we know that
\begin{equation}\label{AfrakA}
\det\mathfrak{a}(\lambda) = c \prod_{j=1}^p a(f_j(\lambda)), \quad c=\prod_{j=1}^p a_j^{2(p-j)} [a_j^{\circ}]^p,
\end{equation}
where, as before, $f_1,\ldots,f_p$ are the $p$ inverse functions of $\Delta(z)$.
This establishes a one-to-one correspondence between the zeros of $\mathfrak{a}$ in $r\bbD$ and the zeros of $a$ in $\{z: |\Delta(z)|<r\}$ (counting with multiplicities). Indeed, if $z_0$ is a zero of $a$ of order $k$, $a(z)=(z-z_0)^k a_0(z)$, then the right-hand side of~\eqref{AfrakA} obtains the factor $\prod_{j=1}^p (f_j(\lambda)-z_0)^k = \operatorname{const} \times (\lambda - \Delta(z_0))^k$. Conversely, if $\lambda_0$ is a zero of $\det\mathfrak{a}$, then~\eqref{AfrakA} $f_j(\lambda_0)$ is a zero of $a$ for some $j=1,\ldots,p$.

Since $r \bbD \subset \{z: |\Delta(z)|<M_{9} r^p \}$ for all large enough $r>r_0$, we can conclude
$$
\nu_a(r) \le \nu_{\det\mathfrak{a}}(M_{9} r^p) \le M_{10} r^{\frac{2p}{\eta}} = M_{10} r^{1-\tau}.
$$

\smallskip

To prove (d), we use integration by parts and part (c) to get the estimate
\begin{multline}\label{sum1/R}
\sum_{j: |R_j|>r} \frac{1}{|R_j|} = \int_r^\infty \frac{d\nu_a(t)}{t} = -\frac{\nu_a(r)}{r} + \int_r^\infty \frac{\nu_a(t)}{t^2} dt
\\
\le M_{10}\int_r^\infty t^{-1-\tau} dt = \tfrac{M_{10}}{\tau} r^{-\tau}.
\end{multline}

\smallskip

To obtain (e), (f), and (g), let us remind to the reader the so-called Case $C_0$ sum rule for MOPRL, see~\cite[Thm 10.2]{DKS}. 
It states that whenever any two out of the three quantities $\mathcal{Z},\mathcal{E}_0,\mathcal{A}_0$ (defined below) are finite, then all of them are finite and
$$
\mathcal{Z}(\mu_\Delta)-\mathcal{E}_0(\mu_\Delta) = \mathcal{A}_0(\Delta(\calJ)),
$$
where
\begin{align*}
\mathcal{Z}(\mu_\Delta) & = \tfrac{1}{2\pi} \int_{-2}^2 (4-x^2)^{-1/2} \log\det\left[ \tfrac{1}{2\pi} \mathfrak{a}(x) \right]^{-1} dx; \\
\mathcal{E}_0(\mu_\Delta) & = \sum_{j=1}^N \log |\beta_j|, \quad \beta_j+\beta_j^{-1}=\Delta(E_j), |\beta_j|>1;  \\
\mathcal{A}_0(\Delta(\calJ)) & = - \sum_{j=1}^\infty \log \det|A_j|
\end{align*}
(we recomputed these to fit our notation; we remark that $E_j$'s here are the eigenvalues of $\calJ$, which means $\Delta(E_j)$ are the eigenvalues of $\Delta(\calJ)$).

By Lemma~\ref{lemMV}(c) and Lemma~\ref{lemCoef}(b), $\mathcal{A}_0(\Delta(\calJ))$ can be uniformly bounded above and below, and by parts (a) and (b) of the current lemma, so can be $\mathcal{E}_0(\mu_\Delta)$. Therefore $\mathcal{Z}(\mu_\Delta) \ge -M_{11}$ is uniformly bounded below. By making the change of variables $x=\Delta(y)$ in  $\mathcal{Z}(\mu_\Delta)$ and using~\eqref{AfrakA}, we obtain
$$
-M_{11} \le \mathcal{Z}(\mu_\Delta) = -\tfrac{1}{2}\log\tfrac{c}{(2\pi)^p} - \tfrac{1}{2\pi} \int_\fre (4-\Delta(y)^2)^{-1/2} \log a(y) |\Delta'(y)| dy.
$$
Note that $4-\Delta(y)^2 = -\tfrac{1}{[a_1^{\circ}\ldots a_p^{\circ}]^p} r(y)$, where $r$ is~\eqref{r} (this follows by noting that both polynomials have the same zeros and the same leading terms).
Using the upper and lower estimates on $a_j$ and $a_j^{\circ}$, we arrive at
\begin{equation}\label{c0rule}
 \int_\fre |r(y)|^{-\frac12} |\Delta'(y)| \log a(y) dy \le  M_{12}.
\end{equation}
Note that $a$ is entire of growth order $<1$, so by the Hadamard factorization, $| a(z) | = A \prod_{j=1}^{N_1} |z-\zeta_j| \prod_{j=1}^{N_2} |1-\frac{z}{\xi_j}|$, where we denoted $\{\zeta_j\}_{j=1}^{N_1}$ ($N_1<\infty$) to be the singularities in $r_0 \bbD$ (repeated according to their multiplicities), and by $\{\xi_j\}_{j=1}^{N_2}$ ($N_2\le \infty$) the rest. Here the constant $A$ is equal to
\begin{equation}\label{normalization1}
A = \int_\fre \frac{\sqrt{|r(x)|}}{\prod_{j=1}^{N_1} |x-\zeta_j| \prod_{j=1}^{N_2} |1-\frac{x}{\xi_j}| } dx + \sum_{k: \zeta_k\in\sigma(\calJ)} 2\pi \frac{\sqrt{|r(\zeta_k)|}}{\prod_{j\ne k} |\zeta_k-\zeta_j| \prod_{j=1}^{N_2} |1-\frac{\zeta_k}{\xi_j}| }.
\end{equation}
We assume that $N_1 \ge 2$ since otherwise (e), (f), (g) are trivial. Let us denote
$$
d\vartheta (y):= |r(y)|^{-\frac12} |\Delta'(y)| dy,
$$
which is a finite measure on $\fre$. Then~\eqref{c0rule} produces
\begin{multline}\label{c0rule2}
M_{12}\ge  \int_\fre \log a(y) \, d\vartheta (y) \ge \log A \int_\fre d\vartheta (y)  \\ 
 + N_1 \inf_{\zeta\in r_0\bbD} \int_\fre  \log |y-\zeta| d\vartheta (y) +  \Big(\sum_{j=1}^{N_2} \log \inf_{x\in\fre}  |1-\tfrac{x}{\xi_j}|\Big) \int_\fre d\vartheta (y).
\end{multline}
Also,
\begin{equation*}\label{c0rule3}
\log A \ge  \log \int_\fre \frac{\sqrt{|r(x)|}}{|(x-\zeta_1)(x-\zeta_2)| } dx - (N_1-2) \log(2r_0) - \sum_{j=1}^{N_2} \log \sup_{x\in\fre}  |1-\tfrac{x}{\xi_j}|.
\end{equation*}
Using part (a), Lemma~\ref{lemTech}(a), and~\eqref{sum1/R}, we obtain
\begin{align}
\label{sumlogs1} \sum_{j=1}^{N_2} \log \inf_{x\in\fre}  |1-\tfrac{x}{\xi_j}| & 
\ge -\sum_{j=1}^{N_2} \frac{r_0}{2|\xi_j|(1-\tfrac{r_0}{2|\xi_j|})} \ge -\sum_{j=1}^{N_2} \frac{r_0}{|\xi_j|} \ge -r_0 \frac{M_{10}r_0^{-\tau}}{\tau};\\
\label{sumlogs2} \sum_{j=1}^{N_2} \log \sup_{x\in\fre}  |1-\tfrac{x}{\xi_j}| & \le \sum_{j=1}^{N_2} \frac{r_0}{2|\xi_j|} \le \frac{r_0}{2} \frac{M_{10}r_0^{-\tau}}{\tau}.
\end{align}

Combining all of this together, while also noting that $\inf_{\zeta\in r_0\bbD} \int_\fre  \log |y-\zeta| d\vartheta (y)$  is finite for any $r_0$, we get
\begin{equation}\label{closeToEdges}
\log \int_\fre \frac{\sqrt{|r(x)|}}{|(x-\zeta_1)(x-\zeta_2)| } dx \le  M_{13}
\end{equation}
for any choice of two zeros $\zeta_1, \zeta_2$ of $a$ in $r_0 \bbD$. We claim that (e) and (f) follow from this. Indeed, using elementary contour integration and some complex analysis, one obtains that for small enough $\delta>0$,
\begin{align}
\label{contourInt1} \int_\fre \frac{\sqrt{|r(x)|}}{|x-(\alpha_j-\delta)|^2 } dx & \ge c_1+c_2\frac{1}{\sqrt{\delta}}, \quad 1\le j \le p ; \\
\label{contourInt2} \int_\fre \frac{\sqrt{|r(x)|}}{|x-(\beta_j+\delta)|^2 } dx & \ge c_1+c_2\frac{1}{\sqrt{\delta}}, \quad 1\le j \le p ; \\
\label{contourInt3} \int_\fre \frac{\sqrt{|r(x)|}}{|x-x_0|^2 +\delta^2} dx & \ge c_1+c_2\frac{1}{\sqrt{\delta}} , \quad x_0\in \cup_{j=1}^p (\alpha_j-\delta,\beta_j+\delta),
\end{align}
where constants $c_1\in\bbR$, $c_2>0$ are independent of $x_0$ (though they can depend on $\fre$). Estimates~\eqref{contourInt1},~\eqref{contourInt2} together with an upper bound on the left-hand side of~\eqref{closeToEdges} show that we can choose $\delta_1>0$ small enough so that no two zeros of $a$ can be in the same $\delta_1$-neighborhood of an endpoint of $\fre$. Similarly, since zeros of $a$ come in complex-conjugate pairs,~\eqref{contourInt3} shows that  there can be no zeros in $W_{\delta_1}(\fre)$ for the chosen $\delta_1$. This proves (e) and (f).

\smallskip

We are left with proving (g). Choose any eigenvalue $E_0$ of $\calJ$. Note that $|E_0|\le r_0/2$ by (a), so $E_0 = \zeta_j$ for some $1\le j\le N_1$. Suppose $\zeta_1\ne E_0$. 
Let us reuse the estimate~\eqref{c0rule2} where we bound the normalization constant $A$~\eqref{normalization1} as
\begin{align*}
\log A &\ge \log \frac{  2\pi \sqrt{|r(E_0)|}}{\prod_{j: \zeta_j \ne E_0} |E_0-\zeta_j| \prod_{j=1}^{N_2} |1-\frac{E_0}{\xi_j}| }
 \\
 &\ge \log \frac{  2\pi \sqrt{|r(E_0)|}}{ |E_0-\zeta_1| } - (N_1-2) \log(2r_0) - \sum_{j=1}^{N_2} \log \sup_{x\in\tfrac{r_0}{2}\bbD}  |1-\tfrac{x}{\xi_j}|.
\end{align*}
Plugging this into~\eqref{c0rule2} and repeating~\eqref{sumlogs2}, we get
$$
\log \frac{ \sqrt{|r(E_0)|}}{ |E_0-\zeta_1| } \le M_{14}.
$$
This shows that if $E_0\notin  U_{\delta_1/2}(\fre)$ then
$$
|E_0-\zeta_1| \ge e^{-M_{14}} \min_{z\in \tfrac{r_0}{2}\bbD\setminus U_{\delta_1/2}(\fre)} \sqrt{|r(z)|},
$$
and if $E_0\in U_{\delta_1/2}(\fre)$, then $|E_0-\zeta_1| > \delta_1/2$ by (e) and (f). Now we can take $\delta_0$ to be $\min\{\delta_1/2, e^{-M_{14}} \min_{z\in \tfrac{r_0}{2}\bbD\setminus U_{\delta_1/2}(\fre)} \sqrt{|r(z)|}\}$.
\end{proof}

Let us now define $\veps_0$ and $R_0$ of our Theorem~\ref{thmStability} to be $\delta_0/2$ and $r_0$ (defined in Lemma~\ref{lemSp}), respectively, and choose any $0<\veps<\veps_0$ and $R>R_0$.

Let $\mu$ and $\tilde{\mu}$ be the spectral measures of $\calJ$ and $\tilde\calJ$ from Theorem~\ref{thmStability}. Let $m_n = \int x^n d\mu(x)$ and $\tilde{m}_n = \int x^n d\tilde{\mu}(x)$ be the $n$-th moments of these measures. 

Singularities of $\calJ$ and $\tilde\calJ$ near the edges will have to be treated with a special care, so let us  change their labeling. Divide all the singularities $\{R_j\}_{j=1}^K \cup \{E_j\}_{j=1}^N$ of $\calJ$ (repeated according to the multiplicities) into those that belong to $U_{\delta_0}$, labeled as $\{\phi_j\}_{j=1}^{N_0}$ ($N_0 \le 2p$ by Lemma~\ref{lemSp}(e) and~(f)); those that belong to $r_0 \bbD \setminus U_{\delta_0}$, labeled as $\{\psi_j\}_{j=1}^{N_1}$ ($N_1\le M_{15} r_0^{1-\tau}$ by Lemma~\ref{lemSp}(c)); and those that belong to $\bbC\setminus r_0\bbD$, labeled as $\{\xi_j\}_{j=1}^{N_2}$ ($N_2 \le \infty$). Let us also define $\sigma_j = 1$ if $\phi_j$ is en eigenvalue and $\sigma_j=0$ otherwise ($1\le j\le N_0$), and $\varsigma_j = 1$ if $\psi_j$ is an eigenvalue and $\varsigma_j=0$ otherwise ($1\le j\le N_1$).

By Theorem~\ref{baxterSuper} and ~\eqref{H1}, the function $a(z)$ is entire of exponential growth $<1$, so by the Hadamard factorization theorem,
$$
| a(z) | = A \prod_{j=1}^{N_0} |z-\phi_j| \prod_{j=1}^{N_1} |z-\psi_j| \prod_{j=1}^{N_2} |1-\frac{z}{\xi_j}|.
$$
By Theorem~\ref{sp}, the spectral measure of $\calJ$ is therefore
$$
d\mu(x) = \frac{\sqrt{|r(x)|}}{ |a(x)|}1_{x\in\fre} dx + \sum_{k=1}^{N_0} \sigma_k \mu_k \delta_{\phi_k} + \sum_{k=1}^{N_1} \varsigma_k w_k \delta_{\psi_k},
$$
where
\begin{align*}
\mu_k & = \frac{2\pi \sqrt{r(\phi_k)} }{ A \prod_{j\ne k} |\phi_k-\phi_j| \prod_{j=1}^{N_1} |\phi_k-\psi_j| \prod_{j=1}^{N_2} |1-\frac{\phi_k}{\xi_j}| } ,\\
w_k & = \frac{2\pi \sqrt{r(\psi_k)} }{ A \prod_{j=1}^{N_0} |\psi_k-\phi_j| \prod_{j \ne k} |\psi_k-\psi_j| \prod_{j=1}^{N_2} |1-\frac{\psi_k}{\xi_j}|   }.
\end{align*}
Let us also define
\begin{equation*}
\breve{\mu}_k : = A \, \mu_k, \quad \breve{w}_k := A \, w_k.
\end{equation*}
Finally note that $A>0$ is the normalization constant that makes $\mu$ into a probability measure, which means
\begin{equation}\label{constA}
A= \int_\fre \frac{\sqrt{|r(x)|}}{ \prod_{j=1}^{N_0} |x-\phi_j| \prod_{j=1}^{N_1} |x-\psi_j| \prod_{j=1}^{N_2} |1-\frac{x}{\xi_j}| } dx + \sum_{k=1}^{N_0} \sigma_k \breve{\mu}_k  + \sum_{k=1}^{N_1} \varsigma_k \breve{w}_k.
\end{equation}

The corresponding quantities for $\tilde\calJ$  (see Theorem~\ref{thmStability})  we will denote by the same letters but with the symbol $\tilde{}\,$. In particular, its singularities will be denoted by $\{\tilde\phi_j\}_{j=1}^{N_0}$, $\{\tilde\psi_j\}_{j=1}^{N_1}$, $\{\tilde\xi_j\}_{j=1}^{\tilde N_2}$. By the conditions of Theorem~\ref{thmStability}, $|\phi_j - \tilde\phi_j|<\veps$ for each $j$, $|\psi_j - \tilde\psi_j|<\veps$ for each $j$, and $|\xi_j - \tilde\xi_j|<\veps$ for those $j$ for which $|\xi_j|\le R$. As a result,  in  general, $\tilde\phi_j\in U_{\delta_0+\veps}(\fre)$ and $\tilde{\psi}_j\in (r_0+\veps)\bbD \setminus U_{\delta_0-\veps}$. Note that $N_0, N_1, \sigma_j, \varsigma_j$ are common for $\calJ$ and $\tilde\calJ$, but $N_2$ and $\tilde{N_2}$ may differ.

\begin{lemma}\label{lemStabNo1}
Suppose that $\phi_j = \tilde{\phi}_j$ for all $1\le j \le N_0$. Then under the conditions of Theorem~\ref{thmStability},
$$
|m_n-\tilde{m}_n| \le M_{16} r_0^n (\veps+R^{-\tau}).
$$
\end{lemma}
\begin{proof}
First of all note that if we can show that
\begin{equation}\label{eqStab1}
\Big| \tilde{m}_n - \frac{A}{\tilde{A}} m_n  \Big| \le \tfrac12 M_{16} r_0^n (\veps+R^{-\tau}),
\end{equation}
then we will immediately obtain
\begin{equation*}
|m_n-\tilde m_n|\le  |m_n| \Big| \tilde{m}_0 - \frac{A}{\tilde{A}} m_0  \Big| + \Big|\tilde{m}_n -  \frac{A}{\tilde{A}} m_n  \Big| \le M_{16} r_0^n (\veps+R^{-\tau}).
\end{equation*}

The left-hand side of~\eqref{eqStab1} can be seen to be equal to
\begin{align}
\label{eqStab2} \Big| \int_\fre \frac{x^n \sqrt{|r(x)|}}{|\tilde{a}(x)|}   \Big( 1 - \prod_{j=1}^{N_1}  \frac{|x-\tilde\psi_j|}{|x-\psi_j|} \prod_{|\xi_j|\le R} & \frac{|1-\frac{x}{\tilde\xi_j}|}{|1-\frac{x}{\xi_j}|} \prod_{|\tilde\xi_j|>R} \Big|1-\frac{x}{\tilde\xi_j}\Big| \prod_{|\xi_j|>R} \frac{1}{|1-\frac{x}{\xi_j}|} \Big) dx \Big.\\
\label{eqStab3}  & \Big. + \sum_{k=1}^{N_0} \sigma_k \tilde\mu_k \phi_k^n \Big(1 - \frac{\breve\mu_k}{\tilde{\breve\mu}_k}\Big) \Big.\\
\label{eqStab4}  & \Big. + \sum_{k=1}^{N_1} \varsigma_k \tilde{w}_k \Big(\tilde\psi_k^n - \frac{\breve{w}_k}{\tilde{\breve{w}}_k} \psi_k^n\Big)
\Big|
\end{align}

In the following computations let us write the shorthand $\exp(\ldots)$ to denote the exponential of the expression in the parentheses that immediately precede the exponential.

We use Lemma~\ref{lemSp}(a),(c),(d) and Lemma~\ref{lemTech}(b) to estimate~\eqref{eqStab2} in the absolute value as
\begin{align*}
\left| \tfrac{r_0}{2}\right|^n  & \sup_{x\in\fre} \left( \sum_{j=1}^{N_1} \Big| 1-\frac{x-\tilde\psi_j}{x-\psi_j} \Big| +  \sum_{|\xi_j|\le R} \Big| 1-  \frac{1-\frac{x}{\tilde\xi_j}}{1-\frac{x}{\xi_j}} \Big| +  \sum_{|\tilde\xi_j|>R} \frac{|x|}{|\tilde\xi_j|} \right. \\
& \qquad \qquad \qquad \qquad \qquad \qquad \qquad \qquad \left. + \sum_{|\xi_j|>R} \frac{|x|}{|\xi_j| \, |1-\tfrac{x}{\xi_j}|}   \right) \exp(\ldots) \\
&\le \left| \tfrac{r_0}{2}\right|^n \left( N_1 \frac{\veps}{\delta_0} + \veps \sum_{|\xi_j|\le R} \frac{1}{|\tilde\xi_j|} + \tfrac{r_0}{2} \sum_{|\tilde\xi_j|> R} \frac{1}{|\tilde\xi_j|} + r_0  \sum_{|\xi_j|> R} \frac{1}{|\xi_j|}  \right) \exp(\ldots)\\
&\le \left| \tfrac{r_0}{2}\right|^n \left( M_{15} r_0^{1-\tau} \frac{\veps}{\delta_0} + \veps M_{15} r_0^{-\tau} + \tfrac{r_0}{2} M_{15} R^{-\tau} + r_0  M_{15} R^{-\tau}  \right) \exp(\ldots)
\\
&\le M_{17} r_0^n (\veps+R^{-\tau}).
\end{align*}

We can estimate the terms~\eqref{eqStab3} and ~\eqref{eqStab4} in a similar fashion. Let us show this for~\eqref{eqStab4} and leave the simpler term~\eqref{eqStab3} as an exercise to the reader. In what follows note that $\tilde{w}_k\le 1$ and that $|\psi_k|\le \tfrac{r_0}{2}$,$|\tilde\psi_k|\le \tfrac{r_0}{2}$ if $\varsigma_k=1$ (Lemma~\ref{lemSp}(a)). We obtain
\begin{align*}
|\eqref{eqStab4}| & \le \sum_{k=1}^{N_1} \varsigma_k |\tilde\psi_k^n-\psi_k^n| + \sum_{k=1}^{N_1}  \varsigma_k |\psi_k|^n \Big| 1-\frac{\breve{w}_k}{\tilde{\breve{w}}_k}\Big| \\
& \le N_1 \veps n \left|\tfrac{r_0}{2}\right|^{n-1}
+ \sum_{k=1}^{N_1}  \varsigma_k \left|\tfrac{r_0}{2}\right|^n \left(  \Big| 1- \frac{\sqrt{r(\psi_k)}}{\sqrt{r(\tilde\psi_k)}}\Big|  + \sum_{j=1}^{N_0} \Big| 1-\frac{\tilde\psi_k-\phi_j}{\psi_k-\phi_j} \Big| \right. \\
& \quad + \sum_{j\ne k} \Big| 1-\frac{\tilde\psi_k-\tilde\psi_j}{\psi_k-\phi_j} \Big|
+ \sum_{|\xi_j|\le R} \Big| 1-  \frac{1-\frac{\tilde\psi_k}{\tilde\xi_j}}{1-\frac{\psi_k}{\xi_j}} \Big| + \sum_{|\tilde\xi_j|>R} \frac{|\tilde\psi_k|}{|\tilde\xi_j|}
\\
& \qquad \qquad \qquad \qquad \qquad \qquad \qquad \qquad
+ \left. \sum_{|\xi_j|>R} \frac{|\psi_k|}{|\xi_j| \, |1-\tfrac{x}{\xi_j}|} \right) \exp(\ldots).
\end{align*}
Observe that
$$
\Big| 1- \frac{\sqrt{r(\psi_k)}}{\sqrt{r(\tilde\psi_k)}}\Big| = \frac{\left| r(\psi_k)-r(\tilde\psi_k)\right|}{\left(\sqrt{r(\psi_k)}+\sqrt{r(\tilde\psi_k)}\right)\sqrt{r(\tilde\psi_k)}} \le M_{18} |\psi_k-\tilde\psi_k|
$$
since $\psi_k,\tilde\psi_k \in \tfrac{r_0}{2}\bbD\setminus U_{\delta_0/2}(\fre)$. Also note that
\begin{multline*}
\sum_{|\xi_j|\le R} \Big| 1-  \frac{1-\frac{\tilde\psi_k}{\tilde\xi_j}}{1-\frac{\psi_k}{\xi_j}} \Big| \le
\sum_{|\xi_j|\le R} \frac{|\tilde\psi_k|\, |\xi_j - \tilde\xi_j| + |\tilde\xi_j|\, |\psi_k-\tilde\psi_k| }{|\xi_j|\, |\tilde\xi_j|\,|1-\tfrac{\psi_k}{\xi_j}|}
\\
\le
\sum_{|\xi_j|\le R} \frac{|\xi_j - \tilde\xi_j|}{|\tilde\xi_j|}
+
\sum_{|\xi_j|\le R} \frac{2|\psi_k-\tilde\psi_k| }{|\xi_j|}
\le \veps r_0^{-\tau} + 2\veps r_0^{-\tau}.
\end{multline*}
Combining this all together and using Lemma~\ref{lemSp}(g) we can conclude that
\begin{align*}
|\eqref{eqStab4}| & \le N_1 \veps |r_0|^n + \sum_{k=1}^{N_1}  \left|\tfrac{r_0}{2}\right|^n
\Big( M_{18}\veps  + N_0\frac{\veps}{\delta_0} \Big.  \quad \\
&
+ (N_1-1) \frac{2\veps}{\delta_0}
+ 3\veps r_0^{-\tau} + \left| \tfrac{r_0}{2}\right| R^{-\tau}
+ \Big. |r_0| R^{-\tau} \Big) \exp(\ldots) \le  M_{19} r_0^n  (\veps+R^{-\tau}).
\end{align*}
This finishes the proof of the lemma.
\end{proof}

%
\begin{lemma}\label{lemHolder}
Let $\delta_0<\min\{\frac{\alpha_{j+1}-\beta_j}{2},\frac{\beta_j-\alpha_j}{2}\}$. There exists a constant $M_{20}>0$ such that for any choice of $\kappa_j\in\{0,1\}$ $(j=1,\ldots,2p)$ and $z_j,\tilde{z}_j\in [\alpha_j-\delta_0,\alpha_j]$ $(j=1,\ldots,p)$, $z_{j+p},\tilde{z}_{j+p}\in [\beta_j,\beta_j+\delta_0]$ $(j=1,\ldots,p)$ satisfying $|z_j-\tilde{z}_j|<\veps$ $(j=1,\ldots,2p)$ the following holds:
\begin{equation}\label{eqHolder}
\int_\fre \sqrt{|r(x)|} \left| \frac{1}{\prod_{j=1}^{2p} |z_j - x|^{\kappa_j}}-\frac{1}{\prod_{j=1}^{2p} |\tilde{z}_j - x|^{\kappa_j}} \right| dx \le M_{20} \sqrt\veps.
\end{equation}
\end{lemma}
\begin{remarks}
1. As we are about to see in the proof, this lemma is equivalent to saying that the function (and $2^p$ of its close relatives)
$$
\int_\fre \frac{\sqrt{|r(x)|}}{\prod_{j=1}^{2p} |z_j - x|^{\kappa_j}} dx
$$
is H\"{o}lder continuous in each variable with the H\"{o}lder exponent $1/2$. This may be a well-known fact, but we provide a proof in case it is not.

2. The exponent $1/2$ in $\sqrt{\veps}$ on the right-hand side of~\eqref{eqHolder} is sharp, as one can see by taking $\kappa_1=1,\kappa_2=\ldots=\kappa_{2p}=0$, $z_1=\alpha_1, \tilde{z}_1=\alpha_1-\veps$ and reusing~\eqref{contourInt1}.
\end{remarks}
\begin{proof}
It is enough to prove the statement for the case when $z_j =\tilde{z}_j$ for all $j\ne l$, and $z_l\ne \tilde{z}_l$. We can also assume $\kappa_l=1$ since otherwise the inequality~\eqref{eqHolder} is trivial. Finally, assume $1\le l\le p$.

Note that as $z_l$ varies over $(\alpha_l-\delta_0,\beta_l]$ and the rest of $z_j$'s are fixed, the function $\prod_{j=1}^{2p} |z_j - x|^{-\kappa_j}$ is monotonically decreasing for all $x\in (\alpha_k,\beta_k)$, $k=1,\ldots,l-1$, and monotonically increasing for all $x\in (\alpha_k,\beta_k)$, $k=l,\ldots,p$. Therefore one can get rid of the absolute values in the left-hand side of~\eqref{eqHolder} and reduce~\eqref{eqHolder} to
$$
|D(z_1,\ldots,z_{l-1},z_l,z_{l+1} \ldots, z_{2p})-D(z_1,\ldots,z_{l-1},\tilde{z}_{l},z_{l+1},\ldots,z_{2p})|\le M_{20} \sqrt\veps,
$$
where
$$
D(z_1,\ldots,z_{2p}):= \sum_{k=1}^p \pm \int_{\alpha_k}^{\beta_k} \frac{\sqrt{|r(x)|}}{\prod_{j=1}^{2p} (z_j - x)^{\kappa_j}} dx,
$$
where the plus/minus signs  depend on $l$ and the configuration of  $\kappa_j$'s.

Let $z_l(\theta) = (1-\theta) z_l + \theta \tilde{z}_l$ with $0\le \theta\le 1$. Note that $|z'_l(\theta)|=|\tilde{z}_l-z_l|<\veps$.
If we denote $D(\theta):=D(z_1,\ldots,z_{l-1},z_l(\theta),z_{l+1} \ldots, z_{2p})$ then we need to estimate $|D(1)-D(0)| = |\int_0^1 \frac{d}{d\theta} D(\theta) d\theta|$. Note that
\begin{equation}\label{Dprime}
\tfrac{d}{d\theta}D(\theta) = \sum_{k=1}^p \mp \int_{\alpha_k}^{\beta_k} \frac{\sqrt{|r(x)|}}{\prod_{j\ne l} (z_j - x)^{\kappa_j}}
\frac{z'_l(\theta)}{(z_l(\theta)-x)^2} dx.
\end{equation}
The $k\ne l$ terms in~\eqref{Dprime} are easy to estimate in the absolute value  since $|z_l(\theta)-x|>\delta_0$ for $x\in[\alpha_k,\beta_k]$, which implies
\begin{multline*}
 \sum_{k\ne l}  \left| \int_{\alpha_k}^{\beta_k} \frac{\sqrt{|r(x)|}}{\prod_{j\ne l} (z_j - x)^{\kappa_j}}
\frac{z'_l(\theta)}{(z_l(\theta)-x)^2} dx \right| \\
 \le  \frac{\veps}{\delta_0^{2p-1}} \sum_{k=1}^p \int_{\alpha_k}^{\beta_k}
\frac{\sqrt{|r(x)|}}{|\alpha_k - x|^{\kappa_k} \,|\beta_k-x|^{\kappa_{k+p}}} dx \le M_{21} \veps.
\end{multline*}
We are left with estimating the $k=l$ term in~\eqref{Dprime}. In what follows we will use the fact that $\sqrt{|r(x)|}$ is H\"{o}lder continuous with the exponent $1/2$ on $[\alpha_l,\beta_l]$, so
$$
\sqrt{|r(x)|}-\sqrt{|r(\alpha_l)|}\le M_{22} \sqrt{x-\alpha_l} \le M_{22} \sqrt{x-z_l(\theta)}.
$$
Let $s_l = \tfrac12 (\alpha_l+\beta_l)$ and split $\int_{\alpha_l}^{\beta_l}$ as $\int_{\alpha_l}^{s_l}+ \int_{s_l}^{\beta_l}$. The $k=l$ term can then be estimated in the absolute value by
\begin{align*}
\left| \int_{\alpha_l}^{\beta_l} \right. & \left.  \frac{\sqrt{|r(x)|}}{\prod_{j\ne l} (z_j - x)^{\kappa_j}}
\frac{z'_l(\theta)}{(z_l(\theta)-x)^2} dx \right| \\
& \le  \frac{\veps  }{\delta_0^{2p-1}} \int_{\alpha_l}^{s_l}
\frac{\sqrt{|r(x)|}-\sqrt{|r(\alpha_l)|}}{|x-z_l(\theta)|^2 } dx + \frac{\veps }{\delta_0^{2p}} \int_{s_l}^{\beta_l} \frac{\sqrt{|r(x)|}}{|\beta_l - x|^{\kappa_{l+p}} } dx
\\
& \le M_{23} \veps \int_{\alpha_l}^{s_l}
\frac{\sqrt{x-z_l(\theta) }}{|x-z_l(\theta)|^{2} } dx + M_{23} \veps
\\
& =
M_{23} \veps \left( \tfrac{2}{\sqrt{\alpha_l-z_l(\theta)}} -\tfrac{2}{\sqrt{s_l-z_l(\theta)}} \right)
 + M_{23} \veps  \le  \frac{2M_{23}\sqrt\veps}{\sqrt{\min\{\theta,1-\theta\}}} + M_{23}\veps,
\end{align*}
where on the last line we used a trivial bound
$$
\alpha_l-z_l(\theta) \ge \min\{|z_l-z_l(\theta)|,|\tilde{z}_l-z_l(\theta)|\} = |z_l-\tilde{z}_l| \min\{\theta,1-\theta\}.
$$

Assuming $\veps<1$ and combining all the terms together we obtain
$$
|D(1)-D(0)| \le \int_0^1 |\tfrac{d}{d\theta} D(\theta)| d\theta \le M_{24} \sqrt{\veps} \int_0^1 \frac{1}{\sqrt{\min\{\theta,1-\theta\}}} d\theta = M_{25}\sqrt{\veps},
$$
which finished the proof.
\end{proof}

\begin{lemma}\label{lemStabNo2}
Suppose that $\psi_j = \tilde{\psi}_j$ for all $1\le j \le N_1$ and $\xi_j = \tilde{\xi}_j$ for all $1\le j \le N_2$. Then under the conditions of Theorem~\ref{thmStability},
$$
|m_n-\tilde{m}_n| \le M_{32} r_0^n (\sqrt{\veps} + R^{-\tau}).
$$
\end{lemma}
\begin{remark}
$\sqrt{\veps}$ cannot be improved.
\end{remark}
\begin{proof}
Just as in Lemma~\ref{lemStabNo1} we want to bound the left-hand side of~\eqref{eqStab1}, which in this case takes the form
\begin{align}
\label{eqStab5} \Big| \int_\fre & \frac{x^n \sqrt{|r(x)|}}{\tilde{A} \prod_{j=1}^{N_1}|x-\psi_j| \prod_{j=1}^{N_2} |1-\frac{x}{\xi_j}| }  \left(  \frac{1}{\prod_{j=1}^{N_0}|x-\tilde\phi_j|} -  \frac{1}{\prod_{j=1}^{N_0}|x-\phi_j|}  \right) dx \Big.\\
\label{eqStab6}  & \qquad \qquad  \qquad \qquad \Big. + \sum_{k=1}^{N_0} \sigma_k \tilde\mu_k(\tilde\phi_k^n-\phi_k^n) + \sum_{k=1}^{N_0} \sigma_k  \phi_k^n \frac{1}{\tilde{A}} \Big( \tilde{\breve\mu}_k  - \breve{\mu}_k\Big) \Big.\\
\label{eqStab7}  &  \qquad \qquad \qquad \qquad \Big. + \sum_{k=1}^{N_1} \varsigma_k \tilde{w}_k \psi_k^n \Big(1 - \frac{\breve{w}_k}{\tilde{\breve{w}}_k} \Big)
\Big|.
\end{align}

First of all note by~\eqref{constA} and~\eqref{sumlogs2}, we get
$$
\tilde{A} \ge  \int_\fre \frac{\sqrt{|r(x)|}}{(2r_0)^{N_1} \prod_{j=1}^{N_2} (1+\frac{r_0/2}{|\xi_j|}) \prod_{j=1}^{N_0} |x-\tilde\phi_j|}  dx \ge M_{26}
$$

Therefore for $x\in\fre$,
\begin{equation*}
 \frac{|x|^n }{ \tilde{A} \prod_{j=1}^{N_1}|x-\psi_j| \prod_{j=1}^{N_2} |1-\frac{x}{\xi_j}| } \le \frac{\left|\tfrac{r_0}{2}\right|^n}{ M_{26} \delta_0^{N_1} \prod_{j=1}^{N_2} (1-\frac{r_0/2}{|\xi_j|}) }  \le M_{27} r_0^n,
\end{equation*}
where on the last step we reused the computation~\eqref{sumlogs1}.

This estimate together with Lemma~\ref{lemHolder} allows us to bound~\eqref{eqStab5} in the absolute value by $M_{20} M_{27} r_0^n \sqrt\veps$.

The first term of~\eqref{eqStab6} is $\le N_0 n \left| \tfrac{r_0}{2} \right|^{n-1} \veps \le M_{28} r_0^n \veps$. The second term of~\eqref{eqStab6} can be bounded in the absolute value by
\begin{multline*}
\sum_{k=1}^{N_0} \sigma_k  \left| \tfrac{r_0}{2} \right|^n M_{26}^{-1}
\frac{2\pi  }{ \prod_{j\ne k} |\tilde\phi_k-\tilde\phi_j| \prod_{j=1}^{N_1} |\tilde\phi_k-\psi_j| \prod_{j=1}^{N_2} |1-\frac{\tilde\phi_k}{\xi_j}| } \\
\times \Big| \sqrt{r(\tilde{\phi_k})}  - \sqrt{r(\phi_k)} \prod_{j\ne k}\frac{|\tilde\phi_k-\tilde\phi_j|}{|\phi_k-\phi_j|} \prod_{j=1}^{N_1} \frac{|\tilde\phi_k-\psi_j|}{|\phi_k-\psi_j|}
\prod_{j=1}^{N_2} \frac{ |1-\frac{\tilde\phi_k}{\xi_j}|}{|1-\frac{\phi_k}{\xi_j}|}  \Big| \\
\le M_{29} \sum_{k=1}^{N_0} \sigma_k  r_0^n  \Bigg[ \Big| \sqrt{r(\tilde{\phi_k})}  - \sqrt{r(\phi_k)} \Big| \Bigg. \\
+ \Bigg. \sqrt{r(\phi_k)} \Big| 1-  \prod_{j\ne k}\frac{|\tilde\phi_k-\tilde\phi_j|}{|\phi_k-\phi_j|} \prod_{j=1}^{N_1} \frac{|\tilde\phi_k-\psi_j|}{|\phi_k-\psi_j|}
\prod_{j=1}^{N_2} \frac{ |1-\frac{\tilde\phi_k}{\xi_j}|}{|1-\frac{\phi_k}{\xi_j}|} \Big| \Bigg].
\end{multline*}
Now note that $\sqrt{r(x)}$ is H\"{o}lder continuous with the H\"{o}lder exponent $1/2$ on $U_{\delta_0}(\fre)\cap (\bbR\setminus \operatorname{Int}(\fre))$, i.e., $\Big| \sqrt{r(\tilde{\phi_k})}  - \sqrt{r(\phi_k)} \Big| \le M_{30} \sqrt{\veps}$; $\sqrt{r(\phi_k)}$ is bounded above on $U_{\delta_0}(\fre)$; and the expression $\Big| 1-  \prod_{j\ne k}\frac{|\tilde\phi_k-\tilde\phi_j|}{|\phi_k-\phi_j|} \prod_{j=1}^{N_1} \frac{|\tilde\phi_k-\psi_j|}{|\phi_k-\psi_j|}
\prod_{j=1}^{N_2} \frac{ |1-\frac{\tilde\phi_k}{\xi_j}|}{|1-\frac{\phi_k}{\xi_j}|} \Big|$ can be estimated by $M_{31} (\veps + R^{-\tau})$ using the same arguments as for~\eqref{eqStab2} in the proof of Lemma~\ref{lemStabNo1}.

This finished an estimate on~\eqref{eqStab6}. Finally,  ~\eqref{eqStab7} can be estimated similarly and is left as an exercise to the reader (note that $\sqrt{r(\psi_k)}$ cancels in the ratio $\frac{\breve{w}_k}{\tilde{\breve{w}}_k}$, which simplifies the estimate).
\end{proof}

\begin{lemma}\label{lemMomToCoeff}
Suppose that for some $L\ge 2$ and $\epsilon>0$, the Jacobi coefficients of $\calJ$ satisfy
\begin{align}
\label{obv1} &|a_n| + |b_n| \le L, \\
&|a_n| \ge 1/L,
\end{align}
and the moments of the spectral measure satisfy
\begin{align}
\label{obv3} &|m_n| \le L^n,\\
\label{obv4} &|m_n -\tilde m_n| \le L^n \epsilon
\end{align}
for all $n$.
Then
\begin{align*}
&|a_n-\tilde{a}_n| \le L^{5(n+1)^2} \epsilon, \\
&|b_n-\tilde{b}_n| \le L^{4n^2} \epsilon
\end{align*}
for all $n$.
\end{lemma}
\begin{remarks}
1. \eqref{obv1} directly implies $||\calJ||\le 3L$ which implies $|m_n|\le (3L)^n$, so~\eqref{obv3} is not really a restriction but rather a convenience in the choice of $L$.

2. The exponents $5(n+1)^2$ and $4n^2$ are by no means optimal. This is not crucial however since the main point of this lemma is to show that we can make the first Jacobi coefficients of $\calJ$ and $\tilde\calJ$ as close as we want by shrinking $\epsilon$ in~\eqref{obv4}.
\end{remarks}
\begin{proof}
Let  $h_n$ ($n\ge 1$) be the Hankel determinant $h_n = \det (m_{j+k-2} )_{j,k=1}^n$. We take $h_0=1$. Let us also define $g_n$ to be the determinant of the same matrix $(m_{j+k-2} )_{j,k=1}^n$ but with the last column $\{m_{n-1},\ldots,m_{2n-2}\}^T$ replaced by $\{m_{n},\ldots,m_{2n-1}\}^T$. Using Heine's formula (see, e.g.,~\cite[Sect~1.7]{OPUC1}), we know that
\begin{align}
\label{heine1} & a_n = \sqrt{\frac{h_{n+1} h_{n-1}}{h_n^2}}, \\
\label{heine2} & b_1+\ldots+b_n = \frac{g_n}{h_n}.
\end{align}

Using the definition of the determinant, Lemma~\ref{lemTech}(c), and~\eqref{obv3}--\eqref{obv4}, we obtain $|h_n-\tilde{h}_n| \le n! n L^{n(n-1)} \epsilon$ and $|g_n-\tilde{g}_n| \le n! n L^{n(n-1)+1} \epsilon$.

Note that~\eqref{heine1} implies that $h_n = \prod_{j=1}^{n-1} a_j^{2(n-j)}$ which gives
\begin{equation}\label{hBound}
L^{-n(n-1)} \le h_n\le L^{n(n-1)}.
\end{equation}

Now apply Lemma~\ref{lemTech}(d) to~\eqref{heine1} with $H=L^{n(n+1)}$, $e=(n+2)! L^{n(n+1)}\epsilon$:
$$
|a_n-\tilde{a}_n|\le 2(n+2)! L^{5n(n+1)} \epsilon \le L^{5(n+1)^2} \epsilon,
$$
where we used an elementary $2(n+2)!\le L^{5n+5}$ for $L\ge2$.

Using~\eqref{hBound} and~\eqref{heine2} we get $|g_n| \le nL^{n(n-1)+1}$. Then
$$
\left| \frac{g_n}{h_n} - \frac{\tilde{g}_n}{\tilde{h}_n} \right| \le \frac{|g_n|\, |h_n-\tilde{h}_n| + |h_n|\, |g_n-\tilde{g}_n| }{|h_n\tilde{h}_n|} \le 2(n+1)!L^{4n(n-1)+1},
$$
which implies
$$
|b_n-\tilde{b}_n| \le 2(n+1)!L^{4n(n-1)+1} \le L^{4n^2}
$$
since $4(n+1)!\le L^{4n-1}$ for $L\ge2$.
\end{proof}

\begin{proof}[Proof of Theorem~\ref{thmStability}]

Applying Lemmas~\ref{lemStabNo1} and~\ref{lemStabNo2} we obtain that there exist uniform constants $M_{33}$ and $r_0$ such that $|m_n-\tilde{m}_n| \le M_{33} r_0^n (\sqrt{\veps} + R^{-\tau})$ for all $n$. An application of Lemma~\ref{lemMomToCoeff} with $L=\max\{M_0,1/\gamma,r_0\}$ and $\epsilon = M_{33}(\sqrt\veps + R^{-\tau})$ finishes the proof.
\end{proof}

\section{Other applications}\label{sApplications}

Our spectral characterization allows us to improve the results of Damanik--Simon~\cite{DS2} and Geronimo~\cite{Geronimo}.

\subsection{$p=1$ and perturbation determinants}\label{ssPerturbationDet}

In this subsection let us restrict ourselves to the case $p=1$. After a shift and a stretch we may put $\fre=[-2,2]$.

Damanik--Simon in~\cite{DS2} studied the problem of classifying all perturbation determinants (equivalently, Jost functions) of~\eqref{expFree} and~\eqref{evFree}. 
The perturbation determinant of $\calJ$ is defined via
$$
L(z)=\det\left( 1 + (\calJ-\calJ^{\circ}) \left[\calJ^{\circ}-(z+z^{-1})\right]^{-1} \right),
$$
where $\calJ^{\circ}$ is the free Jacobi operator. Another way to think about it is as $L(z)=\frac{u(z)}{u(0)}$, where $u(z)$ is the Jost function of $\calJ$ (for more details, see \cite{DS2,KS}). 

Their classification~\cite[Thm 1.9 and Thm 1.11]{DS2} of $L(z)$ involved a non-explicit condition (for the canonical eigenweights to end up being positive). We can improve their result by observing that this condition can be restated in terms of $L$ only, as the oddly interlacing type property of its zeros.

We will state the result for the exponentially decaying perturbations only. For finite range perturbations the result is similar but with $L$ being a real polynomial.

\begin{theorem}\label{DamSim}
Let $R > 1$. Then a function $L(z)$ on $\bbD$ is a perturbation determinant of a Jacobi matrix satisfying
\begin{equation*}
\limsup_{n\to\infty} \left( |1-a_n|+|b_n|\right)^{1/2n} \le R^{-1}
\end{equation*}
if and only if
\begin{itemize}
\item[(i)] $L(z)$ has a real analytic continuation to $\{z : |z| < R\}$;
\item[(ii)] $L(z)$ is non-vanishing on $\overline{\bbD}\setminus\bbR$;
\item[(iii)] all the zeros of $L(z)$ in $\overline{\bbD}\cap\bbR$ are simple;
\item[(iv)] $L(0)=1$;
\item[(v)] Let $0<x_n<\ldots<x_2<x_1<1$ be the positive zeros of $L$ in $\bbD$, and suppose $x_{k+1}\le R^{-1}<x_k$. Then
\begin{itemize}
\item[(a)] There is an even number of zeros $($counting with multiplicities$)$ of $L$ on $[1,x_1^{-1})$;
\item[(b)] There is an odd number of zeros $($counting with multiplicities$)$  of $L$ on $(x_j^{-1},x_{j+1}^{-1})$ $(1\le j\le k-1)$;
\item[(c)] $x_j^{-1}$ is not a zero of $L$ $(1\le j\le k)$;
\end{itemize}
\item[(vi)] Let $-1<y_1<y_2<\ldots<y_n<0$ be the negative zeros of $L$ in $\bbD$, and suppose $y_{k}< -R^{-1}\le y_{k+1}$. Then
\begin{itemize}
\item[(a)] There is an even number of zeros $($counting with multiplicities$)$ of $L$ on $(y_1^{-1},-1]$;
\item[(b)] There is an odd number of zeros $($counting with multiplicities$)$  of $L$ on $(y_{j+1}^{-1},y_j^{-1})$ $(1\le j\le k-1)$;
\item[(c)] $y_j^{-1}$ is not a zero of $L$ $(1\le j\le k)$.
\end{itemize}
\end{itemize}
\end{theorem}

\begin{proof}
Zeros of $L$ coincide with the poles of $M$, where $M(z)=-m(z+z^{-1})$ for $z\in\{z:|z|<1\}$. Therefore the poles of $M$ in $\bbD$ are the preimages of the eigenvalues of $\calJ$ under $z\mapsto z+z^{-1}$, and the poles of $M$ in $\bbC\setminus\bbD$ are the preimages of the resonances. Thus Theorem~\ref{DamSim} follows from Theorem~\ref{sp}.

It should be noted that one can also work this out directly from Damanik--Simon conditions without going through our spectral characterization.
%
%
%
%
\end{proof}

\subsection{Point perturbations of spectral measures}\label{ssPoint}

Using our spectral measure characterization we can easily analyze what happens to the Jacobi coefficients in the class~\eqref{exp} or~\eqref{ev} when we add or remove a point mass to/from the spectral measure.

The proofs are immediate from Theorems~\ref{sp} and~\ref{sp_fin}, and therefore will be omitted.

Let $\calJ=(a_n,b_n)_{n=1}^\infty$ be a Jacobi operator satisfying~\eqref{ev} or~\eqref{exp} for some $1<R\le \infty$. Let $\mu$ be its spectral measure. By the removal of a point mass we mean the perturbation of the form
$$
d\tilde{\mu}(x) = \tfrac{1}{1-w_0} (d\mu(x)-w_0 \delta_{E_0}),
$$
where $w_0 = \mu(\{E_0\})$ (the constant $\tfrac{1}{1-w_0}$ is inconsequential here; its only purpose is to make sure that $\tilde{\mu}$ is a probability measure). By the addition of a point mass we mean a perturbations of the form
$$
d\tilde{\mu}(x) = \tfrac{1}{1+w_0} (d\mu(x)+w_0 \delta_{E_0}),
$$
where $\mu(\{E_0\}) = 0$. We will restrict ourselves to the case $E_0 \notin \esssup \mu$. We denote the Jacobi operator corresponding to $\tilde{\mu}$ by $\tilde{\calJ}=(\tilde{a}_n,\tilde{b}_n)_{n=1}^\infty$. We are interested in the decay properties of the coefficients of $\tilde{\calJ}-\calJ$.

\subsubsection{One interval case $p=1$}

\vspace{0.1cm}

The following facts hold:

a. Removing or adding a point mass is an exponentially decaying perturbation.

b. If $\calJ$ is eventually free~\eqref{evFree}, then removing a point mass is a finite rank perturbation (i.e., $\tilde{\calJ}$ is also eventually free~\eqref{evFree}).

c. If $\calJ$ is eventually free~\eqref{evFree}, then adding a point mass at $E_0$ is a finite rank perturbation if and only if: (1) $\calJ$ has a resonance at $E_0$; (2) $w_0$ is canonical, i.e., determined by~\eqref{weights2}; (3) there is an even number of singularities of $\calJ$ between $(-2,2)$ and $E_0$.

d. If $\calJ$ is eventually free~\eqref{evFree}, but at least one of the three conditions in (c) is violated, then the perturbation is exponentially decaying with the rate of decay
\begin{equation}\label{pointPerturbation1}
\limsup_{n\to\infty} \left(|\tilde{a}_n-a_n|+|\tilde{b}_n-b_n|\right)^{1/2n}= 
\frac{|E_0|-\sqrt{|E_0|^2-4}}{2}.
\end{equation}

\begin{remarks}
1. The expression on the right of~\eqref{pointPerturbation1} is $R^{-1}$, where $R>1$ solves $R+R^{-1}=|E_0|$.


2. One can also state more specifics about the rate of decay in various cases when one adds or removes a point mass for $\calJ$ satisfying~\eqref{expFree}. We leave out the details.
\end{remarks}

\subsubsection{The case $p>1$}

\vspace{0.1cm}

The following facts hold:

a. If $\calJ$ is eventually periodic~\eqref{ev}, then removing a point mass produces $\tilde{\calJ}$ that is also eventually periodic~\eqref{ev}. Unlike for the case $p=1$, this means that the perturbation is generically \textit{not} finite rank, nor even compact. See Remark 1 below.

b. If $\calJ$ satisfies~\eqref{exp} for some $1<R\le \infty$, then removing a point mass produces $\tilde{\calJ}$ that also satisfies~\eqref{exp} with the same $R$ but possibly different $(a_n^{\circ},b_n^{\circ})_{n=1}^\infty$.

c. If $\calJ$ satisfies~\eqref{exp} for some $1<R\le \infty$, then adding a point mass produces $\tilde{\calJ}$ that also satisfies~\eqref{exp} with possibly smaller $R$ and possibly different $(a_n^{\circ},b_n^{\circ})_{n=1}^\infty$. Similarly to the $p=1$ case, one can see that in general $R$ decreases unless $E_0$ is a resonance of $\calJ$, $w_0$ is canonical, and an oddly interlacing type condition holds. We omit the explicit details in hopes that they should be clear from the previous discussion.

\begin{remarks}
1. Let us illustrate property (a) on a simple $p=2$ case. Suppose $\mu$ is the spectral measure of the periodic Jacobi matrix $\calJ$ with $a$ coefficients $a_1,a_2,a_1,a_2,\ldots$ and $b$ coefficients $b_1,b_2,b_1,b_2,\ldots$. Suppose $\calJ$ has an eigenvalue $(z_0)_+$ in the gap. It can be shown that if we remove this eigenvalue then the Jacobi matrix $\tilde{\calJ}$ corresponding to the new measure is going to be the 2-periodic matrix with $a$ coefficients $a_2,a_1,a_2,a_1,\ldots$ and $b$ coefficients $b_1,b_2,b_1,b_2,\ldots$. Clearly such a perturbation is neither finite rank nor even compact.

2. It is amusing that if one considers perturbation of the type $\tfrac{1}{1-w_0} (d\mu(x)-w_0 \delta_{E_0})$ with $w_0 < \mu(\{z_0\})$ (i.e., varying weight but not removing it completely), then  for every such $w_0$ the corresponding perturbation on the Jacobi coefficients $|\tilde{a}_n-a_n|+|\tilde{b}_n-b_n|$ is exponentially decaying. This follows from a very general result of Simon~\cite[Cor 24.4]{S_CD}. However when $w_0 = \mu(\{z_0\})$ the perturbation is no longer exponentially decaying as $\calJ$ becomes exponentially close to a completely different periodic matrix (generically).

3. As we have just seen,  if $\calJ$ does not have a resonance at $E_0\in\calE$ then we cannot add a pure mass at $E_0$ while preserving the rate of exponential convergence. Indeed this would introduce both an eigenvalue and a resonance at $E_0$ and violate $(O_1)$. Similarly, even if $E_0\in\calE$ is a resonance, we can transform it into an eigenvalue only if it is an odd-numbered singularity counting from any gap (see $(O_1)$).

It was noted by Geronimo (in $p=1$ situation, see~\cite[Thm 7]{Geronimo}) that if $E_0 > \sup \supp\,\mu$, then one can add $\{E_0\}$ to the spectrum without changing the exponential decay rate of the coefficients by first dividing the measure (i.e., both the a.c. part and the eigenweights) by a linear factor $(E_0-x)$ and then adding a pure mass at $E_0$ (in fact, this would actually work only if there is an even number of singularities on $[\sup \esssup\,\mu,E_0)$). 
This type of transformation is sometimes referred to as a Christoffel transform. From our point of view analyzing a Christoffel transformation is easy as it just multiplies $a(z)$ in Theorem~\ref{sp} by a linear factor.

In fact, if one wants a procedure to add an eigenvalue that would work for~\eqref{exp} regardless of $E_0 > \sup \supp\,\mu$ or oddly interlacing conditions, then one should use Christoffel transformation with a \textit{quadratic} factor $(E_0-x)(E_0+\veps-x)$ and then add a pure mass at $E_0$. From the point of view of the singularities, this adds an eigenvalue at $E_0$ as well as a resonance at $E_0+\veps$, so it does not ruin $(O_1)$ or $(O_2)$ (but one needs to choose $\veps>0$ or $\veps<0$ depending on the initial configuration of resonances and eigenvalues). From the point of view of the spectral measure, this always preserves positivity if $\veps$ is small enough. The proof is an immediate application of Theorem~\ref{sp}, which allows to avoid the cumbersome orthogonal polynomials computation of Nevai or the double commutator method of Gesztesy--Teschl. We stress that this works for any $p\ge 1$ and can be applied in the gaps, however it is restricted to our class~\eqref{exp} or~\eqref{ev}.
\end{remarks}

\bibliographystyle{plain}
\bibliography{../../mybib}

\end{document}